\documentclass[11pt,reqno, a4paper]{amsart}
\usepackage{marginnote}

\usepackage[utf8]{inputenc}
\usepackage{amsmath,textcomp,  microtype, lmodern, mathrsfs,csquotes,fbb}
\usepackage{todonotes}
\usepackage{amsmath}
\usepackage{amsfonts}
\usepackage{mathrsfs}  
\usepackage{amssymb}
\usepackage{amsthm}
\usepackage{graphicx}
\usepackage[noadjust]{cite}
\usepackage{caption}
\usepackage{subfig}
\usepackage{dutchcal}
\usepackage{bbm}
\usepackage{tikz}
\usetikzlibrary{decorations.pathreplacing}
\usepackage[T1]{fontenc}
\usepackage{mathtools}
\usepackage{pdfpages}
\usepackage[top=2.92cm, bottom=2.92cm, left=2.5cm, right=2.5cm, headsep=0.2in]{geometry}
\usepackage{fancyhdr}

\DeclareMathOperator{\ind}{Index}
\theoremstyle{plain}
\newtheorem{theorem}{Theorem}[section]

\newtheorem{corollary}[theorem]{Corollary}
\newtheorem{Lemma}[theorem]{Lemma}

\newtheorem{remark}[theorem]{Remark}

\newtheorem{proposition}[theorem]{Proposition}
\newtheorem{obsv}[theorem]{Observation}
\newtheorem{conjecture}[theorem]{Conjecture}

\newcommand{\RR} {\mathbb R}

\newcommand{\CC} {\mathbb C}
\newcommand{\DD} {\mathbb D}
\newcommand{\SSS} {\mathbb S}

\newcommand{\ZZ} {\mathbb Z}
\newcommand{\NN} {\mathbb N}
\newcommand{\TT} {\mathbb T}

\newcommand{\NNN}{\mathcal{N}}

\newcommand{\CCC}{\mathcal{C}}

\newcommand{\Pot}{V}
\newcommand{\Ht}{H_t}

\newcommand{\pa} {\partial}
\newcommand{\Cal} {\mathcal}

\newcommand{\beq} {\begin{equation}}
\newcommand{\eeq} {\end{equation}}

\newcommand{\inrad}{\operatorname{inrad}}

\numberwithin{equation}{section}

\usepackage[colorlinks=true,
  linkcolor=blue!50!green,
  citecolor=blue!50!green,
  urlcolor=blue!50!green]{hyperref}

\usepackage[nameinlink,capitalise,noabbrev]{cleveref}

\theoremstyle{plain}
\newtheorem{mainthm}{Theorem}

\crefname{mainthm}{Theorem}{Theorems}
\Crefname{mainthm}{Theorem}{Theorems}

\begin{document}
\title[Nodal domains and perturbations]{
Nodal Domains on Surfaces under Perturbation: Upper Semicontinuity, Courant-Sharpness, and Boundary Intersections
}

\author{Saikat Maji}
\address{Indian Institute of Technology Bombay, Powai, Maharashtra 400076, India}
\email{saikat.maji@iitb.ac.in}
\author{Mayukh Mukherjee}
\address{Indian Institute of Technology Bombay, Powai, Maharashtra 400076, India}
\email{mathmukherjee@gmail.com}
\author{Soumyajit Saha}
\address{Institut de Recherche Math\'{e}matique Avanc\'{e}e (IRMA), Strasbourg 67084, France}
\email{soumyajit.saha@unistra.fr}

\begin{abstract}
We study how the number of nodal domains of eigenfunctions of Schr\"odinger operators $-\Delta_{g_t}+V_t$ on closed surfaces changes under smooth perturbations of $(g_t,V_t)$ along convergent eigenbranches. Locally, near each nodal critical point of the limit eigenfunction, we give a sector/graph count showing that no new local domains can be created and that vanishing orders cannot increase. Globally, we prove upper semicontinuity of the nodal domain count; in the noncritical case the count is stable. The result is \emph{branch-free} on spectral clusters. At the wavelength scale, new closed nodal loops cannot be created. 

We also treat \emph{localised} (topology-changing) perturbations: the count inside the unperturbed core cannot increase. As applications, we construct metrics on any closed surface that are Courant-sharp up to an arbitrary finite level and prescribe $2n_i$ boundary intersections on each boundary component. An appendix records a uniform (wavelength-scale) lower bound on the inner radius of nodal domains along the branch.
\end{abstract}

\maketitle
\tableofcontents

\section{Introduction and main results} Let $(M, g)$ be a closed Riemannian surface.  Consider the eigenvalue equation
\begin{equation}\label{eqtn: Eigenfunction equation}
    -\Delta \varphi = \lambda \varphi, %\hspace{2mm} i=1, 2, \cdots,
\end{equation}
where $\Delta$ is the Laplace-Beltrami operator given by (using the Einstein summation convention) 
$$
\Delta f = \frac{1}{\sqrt{|g|}}\pa_i \left( \sqrt{|g|}g^{ij}\pa_j f\right),
$$
where $|g|$ is the determinant of the metric tensor $g_{ij}$. In the Euclidean space, this reduces to the usual $\Delta = \pa_1^2 + \dots + \pa_n^2$ (note that we are using the analyst's sign convention for the Laplacian, namely that $-\Delta$ is positive semidefinite).

Let  $\mathcal{N}(\varphi) = \{ x \in M: \varphi (x) = 0\}$ denote the nodal set of the eigenfunction $\varphi$. On a closed $n$-manifold, the nodal set is the union of a smooth hypersurface and a set that is countably $(n - 2)$-rectifiable (see \cite{Cheng1976, HardtSimon1989, Baer1997}). In particular, for $n =2$, the nodal set is a finite embedded graph, possibly with embedded circle components; away from finitely many nodal critical points, it is a finite union of smooth arcs.
Recall that any connected component of $M \setminus \mathcal{N}(\varphi)$ is known as a nodal domain of the eigenfunction $\varphi$, and we generally denote nodal domains by $\Omega$ (with minor abuse of notation). Let us denote by $\nu(\varphi)$ the number of nodal domains for $\varphi$ and by $\mathcal{S}(\varphi)$ the set of self-intersection points (termed as \textit{nodal critical points}) of $\varphi$. The latter is contained in $\mathcal{C}(\varphi)$, the set of all critical points of $\varphi$. Observe that the set $\Cal{C}(\varphi)$ does not need to be isolated; in fact, it can be a hypersurface, as given by the eigenfunction $\varphi(x,y)=\sin(2k\pi x)$ on the flat torus $\TT^2$. However, $\nabla \varphi$ satisfies an elliptic system for which one can establish Carleman estimates and unique continuation (see, for example, \cite{Laurent2012}), due to which $\Cal{C}(\varphi)$ cannot have an interior point. It follows from the maximum principle that eigenfunctions are either strictly positive or strictly negative on each nodal domain, and change signs across non-critical nodal points.  

\subsection*{Standing framework (Schr\"{o}dinger operators).}
Throughout, we allow a (possibly $t$-dependent) smooth potential $\Pot_t\in C^\infty(M)$ and work with the Schr\"{o}dinger operator
\[
\Ht\ :=\ -\Delta_{g_t}+\Pot_t,
\]
considered as a self-adjoint operator on $L^2(M,d\mathrm{vol}_{g_t})$. Its spectrum is discrete and can be listed, with multiplicity,
\[
\lambda_1(t)\ \le \ \lambda_2(t)\ \le \cdots \nearrow \infty,
\]
with corresponding real-valued $L^2$-normalised eigenfunctions $\varphi_{k,t}$:
\begin{equation}\label{eq:Schr-eig}
  \Ht\varphi_{k,t}=\lambda_k(t)\varphi_{k,t}.
\end{equation}
We assume that $t\mapsto g_t$ and $t\mapsto\Pot_t$ are $C^\infty$ families and, whenever needed locally, that $\|\Pot_t\|_{C^2}$ is bounded uniformly for $|t|\le t_0$.
All notions of nodal set $\mathcal N(\varphi)$, nodal domains, critical and nodal critical points are as before. We note that the critical-point condition $d\varphi(p)=0$ is metric-independent; however, gradient-based quantities such as the index of $\nabla\varphi$ at a zero involve a choice of metric. Throughout the perturbative arguments in Sections~\ref{Sec: Preliminary results}-\ref{Sec: Proofs of main theorem 2}, all gradients and vector-field norms are computed with respect to the fixed background metric $g_0$. This is legitimate because $g_t$ and $g_0$ are uniformly equivalent for small $|t|$, so $\nabla^{g_0}\varphi_t$ and $\nabla^{g_t}\varphi_t$ have the same zero set, and the index at a zero is a topological invariant (homotopy-invariant winding number) that does not depend on this choice.

\medskip
\noindent\textbf{Regularity of nodal sets (2D).} In our Schr\"{o}dinger setting the coefficients are smooth, so the eigenfunctions themselves are $C^\infty$ by standard elliptic regularity. For the \emph{nodal-set structure} in dimension two, which requires only sufficiently regular coefficients, we rely on Hartman--Wintner \cite{HartmanWintner1953} and Alessandrini \cite{alessandrini1992}: the nodal set is a finite embedded graph (possibly with embedded circle components); away from finitely many isolated \emph{nodal critical points} it is a finite union of $C^{1,\alpha}$ arcs meeting at the nodal critical points with equal opening angles. Vanishing orders are integers and the local model is a harmonic polynomial of degree $k$ (no real-analyticity of the coefficients is required).

\medskip
\noindent\textbf{Small-domain obstruction for Schr\"{o}dinger operators.}
The following observation is used repeatedly throughout the paper to exclude the creation of tiny nodal domains under perturbation. We state it here for easy reference.

\begin{Lemma}[Small-domain eigenvalue obstruction]\label{lem:small-domain-Schro}
Let $D\subset (M,g)$ be a domain and $u\not\equiv 0$ a solution of $(-\Delta_g+\Pot)u=\lambda u$ in $D$ with $u|_{\partial D}=0$. Then
\[
  \lambda\ \ge\ \lambda_1^\Delta(D;g)+\inf_D \Pot,
\]
where $\lambda_1^\Delta(D;g)$ is the first Dirichlet eigenvalue of $-\Delta_g$ on $D$.
In particular, combining with the Faber--Krahn inequality $\lambda_1^\Delta(D;g)\ge C_{\mathrm{FK}}/\mathrm{Area}_g(D)$ (uniform for $D$ contained in small coordinate balls and under $C^2$-bounded perturbations of $g$),
\[
  \frac{C_{\mathrm{FK}}}{\mathrm{Area}_g(D)}\ \le\ \lambda - \inf_D\Pot.
\]
Hence if $\|\Pot_t\|_{L^\infty}$ is uniformly bounded and $\lambda_t$ stays in a compact interval, no nodal domain $D_t$ of $\varphi_t$ can be contained in a sufficiently small ball. Quantitatively, if $\lambda_t-\inf_{D_t}\Pot_t\le\Lambda$, then any ball containing $D_t$ has radius $r\ge c/\sqrt{\Lambda}$ for a uniform constant $c>0$.
\end{Lemma}
\begin{proof}
By the min-max principle, $\lambda_1^{\Delta+\Pot}(D)\ge \lambda_1^\Delta(D)+\inf_D\Pot$. Since $u$ vanishes on $\partial D$, the Rayleigh quotient gives $\lambda_1^{\Delta+\Pot}(D)\le\lambda$, yielding the chain.
\end{proof}

One of the intriguing areas of research in spectral geometry is the study of nodal domain counting. This topic has been extensively explored over the past century, with one of the earliest significant results due to Courant \cite{Courant1923} where he established an upper bound on the number of nodal domains associated with an eigenfunction $\varphi_k$, namely,  $\nu(\varphi_k)\leq k$. Shortly thereafter, Stern \cite{Ste} (see also \cite{BH}), in a somewhat opposite vein, demonstrated the existence of a sequence of eigenfunctions on a rectangle, each containing exactly two nodal domains. A similar result was also proved  in \cite{Lewy} where it was shown that there exists a sequence of eigenfunctions $\{\varphi_k\}$ with $\lambda_k\to \infty$ on the usual round sphere $S^2$ such that $\nu(\varphi_k)\leq 3$ for every $k$. The results of Lewy and Stern focus on the construction of eigenfunctions with the minimal possible number of nodal domains.  We also have results that explore the opposite direction. In several special arithmetic or symmetric negatively curved settings, and sometimes only along density-one subsequences, one can prove that the number of nodal domains grows with the eigenvalue; see \cite{ghosh2013nodal, jung2016number, jung2018quantum}. So a naturally related question in this area is the study of Courant-sharp eigenvalues and eigenfunctions. We refer to $\varphi_k$ as Courant-sharp if $\nu(\varphi_k)=k$. Pleijel \cite{Pleijel1956} showed that in the Dirichlet case for domains in $\RR^2$, only finitely many eigenfunctions can be Courant-sharp. This result was further generalised to compact Riemannian manifolds in higher dimensions in the works of \cite{peetre1957generalization, berard1982inegalites}; see also \cite{polterovich2009pleijel} for the Neumann (free membrane) analogue. There is also some previous work towards finding formulae for nodal domain counting which are roughly based on Euler characteristic arguments; for example, see \cite{Hoffmann-OstenhofMichorNadirashvili1999} for the case of planar bounded domains with smooth boundary.  These results, among many others, highlight the intricacies of nodal domain counting, a topic that incorporates a wide range of mathematical tools and techniques in its study.

In this article, our aim is to study the effect of metric perturbations on nodal domain count. It is known that nodal sets are not stable under generic perturbations (see e.g. \cite{Uhlenbeck1976}). Thus it is expected that the number of nodal domains would also change under perturbations.  
In \cite{BeckGuptaMarzuola2024-rectangle}, precise descriptions of nodal structure upon perturbation of rectangular domains were given under simplicity assumptions; in \cite{Lyons2025-rectangle}, the case of eigenvalues with multiplicity $2$ was treated. Depending on the spectral regime, the nodal domain count may drop or remain invariant. To the best of our knowledge, the behaviour of the number of nodal domains for general manifolds under arbitrary perturbations has not been discussed previously, and our investigations are geared towards this. From well-known numerics and also from a heuristic standpoint, it is believed that the nodal intersection picture becomes generally less complicated  on %{\em non-generic} 
perturbation; this is because (heuristically speaking), on perturbation the nodal set ``untwines''. In the background of our considerations is the following fact (see \cite{Uhlenbeck1976,Albert1973}): a nonconstant Laplace eigenfunction for a generic metric (in the sense of Baire category) is a Morse function and also has no nodal critical points. So, generically, the nodal set is a collection of embedded circles, and every critical point is either a point of extremum or a saddle point.

We work with a $C^\infty$ family $(g_t,\Pot_t)$ and a corresponding smooth one-parameter family of eigenpairs
\[
\Ht\varphi_t=\lambda(t)\varphi_t,\qquad \|\varphi_t\|_{L^2(g_t)}=1,
\]
such that, for a fixed branch,
\begin{equation}\label{eq:eigenf_conv_S}
\varphi_t \longrightarrow \varphi_0 \quad \text{in } C^\infty\text{-topology, as } t\to 0,
\end{equation}
and $\lambda(t)\to\lambda(0)$. (This holds automatically near simple spectrum; for clusters we use Riesz projections; for more details on this, see the very general formulation outlined in \cite{KriegenMichorRainer11}).

Our first main result discusses the local behaviour of nodal critical points and nodal domains under small smooth perturbations. Among other things, our study leads to  an estimate on the number of ``local nodal domains'' around a nodal critical point after perturbation. Formally, this means the number of connected components $F_t(p)$ of $\displaystyle(M \setminus \NNN(\varphi_t) )\cap B_{g_0}(p, r)$, where $p$ is a nodal critical point of $\varphi_0$. We spell out clearly that for all our perturbative results, distances and other metric properties are with respect to a fixed background metric which is ``close'' to $g_t$ for small $t$. In particular, this background metric can be $g_0$.

\begin{mainthm}\label{thm:loc_nod_est}  
Let $p\in M$ be a nodal critical point of the Schr\"{o}dinger eigenfunction $\varphi_0$ with order of vanishing $k$. Also, denote by $C_t$ the number of connected components of $\NNN(\varphi_t)\cap B(p, r)$. %Then we have the following. 
Given $r$ small enough, there exists $\varepsilon := \varepsilon(r) > 0$ such that for any $|t| < \varepsilon$, the following hold:
\begin{enumerate}
    \item[(a)] The number of nodal critical points of $\varphi_t$ contained inside $B(p,r)$ does not exceed $k-1$.
    
    \item[(b)] The order of vanishing of any nodal critical point of $\varphi_t$ within $B(p, r)$ cannot be greater than $k$.
    
    \item[(c)]
    We have the following estimate:
    $$
    F_t(p) = F_0(p)-C_t+1.
    $$
    In particular, $F_0(p) \geq F_t(p) $.
\end{enumerate}

\end{mainthm}

Whenever \eqref{eq:eigenf_conv_S} holds for Schr\"{o}dinger eigenbranches, we have the following global version. 

\begin{mainthm}\label{thm:main_result_A}
    The nodal domain count is upper semicontinuous at $t=0$; equivalently, perturbation cannot increase the number of nodal domains. More precisely, whenever \eqref{eq:eigenf_conv_S} holds, we have that
    \beq\label{ineq:nodal_domain_decrease}
     \nu(\varphi_t) \leq \nu(\varphi_0) 
    \eeq 
    for small enough $|t|$. In particular, if $\varphi_0$ has no nodal critical point, then $\nu(\varphi_t) = \nu(\varphi_0)$.
\end{mainthm}
It is interesting to observe that in the above set up, the difference $\nu(\varphi_0) - \nu(\varphi_t)$ is controlled from above by the indices of $\nabla\varphi_0$ and $\nabla \varphi_t$ at their nodal critical points. Such an estimate has been outlined in Corollary \ref{cor:nod_dom_dec_est}, which is complementary to \eqref{ineq:nodal_domain_decrease} above.

Observe that until now we have been dealing with metric perturbations, but the topology stayed the same throughout the perturbation. In Section \ref{sec:changing_topology}, we prove  Theorem \ref{thm:perturb_nodal_domain_reduction}, which is a counterpart of Theorem \ref{thm:main_result_A} above, but in % Now, we make a few comments addressing 
the case of perturbations which are local in nature but allow for change in the topology of the surface. These types of perturbation have been studied extensively (see \cite{Takahashi2002,Kom2005,MukherjeeSaha2022} and references therein). This has interesting consequences. For instance, as an application of  Theorem \ref{thm:perturb_nodal_domain_reduction}, 
we prove that on any given closed surface, there exist Riemannian metrics which support  Courant-sharp Laplace eigenfunctions up to arbitrary (but of course finite) level. We refer the reader to \cite{BerardHelffer2016-Courant_sharp_1,  Bonnaillie-NoelHelffer2016, BerardHelfferKiwan-Courant_sharp_2, BerardHelfferKiwan2022-Courant_sharp_3} for previous results on more specific manifolds. We state the result formally.

\begin{mainthm}\label{thm:Freitas_app}
     Given any two positive integers $k, m$, one can find a closed Riemannian surface $(M, g)$ of genus $m$ such that %the first $k$ eigenvalues are simple, and 
     for every $l\leq k$ the corresponding Laplace eigenfunction $\varphi_l$ of $-\Delta_g$ has exactly $l$ nodal domains, i.e., $\nu(\varphi_l)=l$ for all $l\leq k$. 
\end{mainthm}

Now we look at our final result, which deals with the prescription of the number of nodal intersections on the boundary of a surface. 

\begin{mainthm}\label{thm: prescription_nodal_intersection_multiple_bdry}
    Let $M$ be a Riemannian surface with genus $m$ and boundary components $B_1, \cdots, B_b$, which are topological circles. Given $n_1, \cdots, n_b \in \mathbb{N}$, there exist a metric $g$ on $M$ and an integer $l\in \NN$ such that, for each $i\in \{1, \cdots, b\}$, the nodal set of the Neumann eigenfunction $\varphi_{l+i}$ intersects the boundary component $B_i$ at exactly $2n_i$ points.
\end{mainthm}

Beyond Theorems \ref{thm:loc_nod_est}, \ref{thm:main_result_A}, \ref{thm:perturb_nodal_domain_reduction}, \ref{thm:Freitas_app} and \ref{thm: prescription_nodal_intersection_multiple_bdry}, we establish or discuss the following further results:
\begin{enumerate}
  \item \emph{Branch-free upper semicontinuity on clusters} (Theorem~\ref{thm:cluster_USC}): for eigenfunctions taken from a fixed spectral cluster, $t\mapsto \sup\{\nu(u)\}$ is upper semicontinuous at $t=0$ (no simple-branch selection needed). 
  
  \item \emph{Inner-radius stability} (Appendix, Theorem~\ref{thm: Inner_radius_perturbation}): the minimum inner radius of nodal domains along a fixed eigenbranch obeys a uniform lower bound $\gtrsim \lambda^{-1/2}$ under the perturbations considered.
  \item \emph{Higher-dimensional reduction and conjecture} (Section~\ref{sec:higher_dim}, Proposition~\ref{prop:abstract_USC_highdim} and Conjecture~\ref{conj:isolated_conical_USC}): we prove that regular-value stability holds in all dimensions, isolate a purely combinatorial ``chamber coarsening'' reduction that is dimension-free, and formulate a conjecture for isolated conical singularities under blow-up uniqueness and frequency pinching.
\end{enumerate}
We also recall that Payne-type stability for the second Dirichlet eigenfunction on convex domains was established in \cite{MukherjeeSaha2022}.

\subsection{Organisation of the paper}\label{subsec:plan} In Section~\ref{Sec: Preliminary results}, we prove technical lemmata on stability of indices and nodal critical points, and establish that in the absence of nodal critical points, the nodal domain count is unchanged under perturbation (Proposition~\ref{prop:nodal_line_no_crossing}). In Section~\ref{Sec: Perturbation and critical points}, we prove the boundary intersection count invariance (Lemma~\ref{lem:nod_lin_disc_invariant}) and our first main result, Theorem~\ref{thm:loc_nod_est}. In Section~\ref{Sec: Proofs of main theorem 2}, we prove our second main result, Theorem~\ref{thm:main_result_A}, via the dual incidence graph. In Section~\ref{sec:quant_branchfree}, we prove the branch-free cluster USC (Theorem~\ref{thm:cluster_USC}), the wavelength rigidity result (Theorem~\ref{thm:meso_no_loops}), and the openness of Courant-sharpness (Theorem~\ref{thm:CS_open}). In Section~\ref{sec:changing_topology}, we treat localised (topology-changing) perturbations and prove Theorems~\ref{thm:Freitas_app} and~\ref{thm: prescription_nodal_intersection_multiple_bdry}. In Section~\ref{sec:higher_dim}, we prove a dimension-free regular-value stability theorem and a higher-dimensional USC reduction via chamber coarsening (Propositions~\ref{prop:regular_value_all_dim} and~\ref{prop:abstract_USC_highdim}), and formulate a conjecture for isolated conical singularities (Conjecture~\ref{conj:isolated_conical_USC}). In the Appendix, we prove a uniform inner-radius lower bound $\gtrsim\lambda^{-1/2}$ for nodal domains along the eigenbranch.

\section{Stability on surfaces: preliminaries and local index/sector calculus}\label{Sec: Preliminary results}
\subsection{Stability of isolated critical points under perturbation} 
Let $\mathfrak{X}^1(M)$ denote the Banach space of $C^1$-vector fields on $M$. Let us denote the set of all zeros of $X\in \mathfrak{X}^1(M)$ by $\CCC$. For a vector field $X$ and an {\em isolated zero} $p\in \CCC$, we use $\ind_p(X)$ to denote the index of the vector field $X$ at $p$.  Later we shall extend the notion of index to compact subsets of zeros of $X$ that admit \emph{isolating neighbourhoods}. We call an open set $U\subset M$ an \emph{isolating neighbourhood} of a compact subset $\tilde{\CCC}\subset\CCC$ if $U$ is diffeomorphic to the Euclidean disc, $\tilde{\CCC}\subset U$, and $X\neq 0$ on $\overline{U}\setminus\tilde{\CCC}$. In particular, $\CCC\cap\partial U=\emptyset$. We use the fact that in dimension $2$, the index of an isolated zero can be measured by the winding number of the curve $X \circ \gamma$ around $0$ where $\gamma$ is any positively oriented smooth parametrisation of $\pa U$.
We have the following proposition:
  \begin{proposition}\label{prop: winding number is constant}
    Let $X_t \in \mathfrak{X}^1(M)$ be a 1-parameter family of vector fields on $M$ such that the map $t \mapsto X_t \in \mathfrak{X}^1(M)$
    is continuous at $t_0 \in \RR$ (in the $C^1$-topology).
    Let $p \in M$ be an isolated zero of $X_{t_0}$ and $U$ be an isolating neighbourhood of $p$. For a positively oriented smooth parametrisation $\gamma$ of $\pa U$,  the winding number $w(t)$ of $X_t \circ \gamma$ around $0$ is continuous at $t_0$, and hence locally constant.
\end{proposition}
\begin{proof} 
     We work in a local trivialisation of $TM$ near $p$, identifying vector fields with $\CC$-valued functions. It is sufficient to prove that the function
    \begin{equation*}
        t \mapsto w(t)=\frac{1}{2\pi i} \oint_{X_t \circ \gamma} \frac{dz}{z}
    \end{equation*}
   is continuous at $t_0$, where $\gamma$ is a positively oriented parametrisation of $\pa U$. Let $\eta_t: \SSS^1 \to \CC \setminus \{0\}$ be defined by $\eta_t= X_t \circ \gamma$; for each $t$, $\eta_t$ is $C^1$ in $s$, and $t\mapsto \eta_t$ is continuous in $C^1(\SSS^1)$ near $t_0$. Then the above integral can be re-written as 
   \begin{equation*}
       w(t) = \frac{1}{2\pi i}\int_0^1 \frac{\eta_t' (s)}{\eta_t(s)}ds.
   \end{equation*}
   Since $t\mapsto X_t$ is continuous in $C^1$ near $t_0$ and $X_{t_0}\neq 0$ on the compact set $\gamma(\SSS^1)\subset\pa U$, there exist $m,M,M_1>0$ such that $0<m\le|\eta_t(s)|\le M$ and $|\eta_t'(s)|\le M_1$ for all $s\in\SSS^1$ and all $t$ sufficiently close to $t_0$. Therefore $\eta_t'/\eta_t\to\eta_{t_0}'/\eta_{t_0}$ uniformly on $\SSS^1$, giving $w(t)\to w(t_0)$. Since $w(t)\in\ZZ$, continuity at $t_0$ implies that $w$ is locally constant near $t_0$.
\end{proof}
 Let $p_1, \dots, p_n$ be a finite collection of isolated zeros of a $C^1$ vector field $X$, all contained in an isolating neighbourhood $U$. Then, the {\em total index} of $X$ at the set $\tilde{\CCC}:=\{p_i:i=1,\cdots,n\} \subset \CCC$ is defined to be the sum $\sum_{j = 1}^n \ind_{p_j} (X)=: \ind_{\tilde{\CCC}}(X)$.  The total index equals the winding number of $X \circ \gamma$ around $0$, where $\gamma$ is any positively oriented smooth parametrisation of $\pa U$. Indeed, choose pairwise disjoint small discs $D_i\subset U$ with $p_i\in D_i$ and positively oriented boundaries $\gamma_i$; by excision/additivity of the degree (the map $X/|X|:\partial U\to S^1$ is homotopic, through nonvanishing maps, to the concatenation of $X/|X|:\partial D_i\to S^1$ via paths in $U\setminus\bigcup_i D_i$ where $X\neq 0$), the winding number of $X\circ\gamma$ equals the sum of the winding numbers of $X\circ\gamma_i$, each of which equals $\ind_{p_i}(X)$.

The set of zeros of a vector field is not necessarily discrete in general, e.g.\ when $X=\nabla \varphi$ for a Laplace eigenfunction $\varphi$ (as mentioned in the introduction). We cannot define the index of a vector field at a set of non-isolated zeros by simply summing over individual indices. However, \cite[Section~1.1.2]{Brasselet87} provides a definition of the total index of a vector field at any compact subset of zeros that admits an isolating neighbourhood with smooth boundary. We quote the key result:
\begin{theorem}\cite[Theorem 1.1.2.]{Brasselet87}\label{thm: extended vector field -total index}
Let $U \subset M$ be a domain with smooth boundary $\pa U$ and $X\in \mathfrak{X}^1(M)$ such that $X \neq 0$ in a neighbourhood $V$ of $\pa U$. Then
\begin{enumerate}
    \item The restriction of $X$ to $V$ can be extended to a vector field $\tilde X$ on $U$ with finitely many isolated zeros $p_1,\dots,p_r$.
    \item The total index $\sum_{j=1}^r\ind_{p_j}(\tilde X)$ is independent of the choice of extension. In other words, if $\tilde{X}_1$ and $\tilde{X}_2$ are two such extensions with corresponding finite sets of isolated zeros $\tilde{\mathcal{C}}_1$ and $\tilde{\mathcal{C}}_2$ respectively, then $\ind_{\tilde{\mathcal{C}}_1}(\tilde{X}_1)=\ind_{\tilde{\mathcal{C}}_2}(\tilde{X}_2)$.
\end{enumerate}
\end{theorem}

Consider a connected component $\tilde{\mathcal{C}}$ of the set of zeros $\mathcal{C}$ of the vector field $X$. 
Let $\mathcal{T}$ be the \textit{cellular tubular neighbourhood} of $\tilde{\mathcal{C}}$ with smooth boundary as defined in \cite[Section 1.1.2]{Brasselet87}. We note that for any open neighbourhood $U$ of $\tilde{\mathcal{C}}$, there exists a cellular tubular neighbourhood $\mathcal{T}$ such that $\tilde{\mathcal{C}}\subset \mathcal{T} \subset U$. We always assume that a neighbourhood $U$ of a set of zeros is tubular and it satisfies $\mathcal{C}\cap U = \tilde{\mathcal{C}}$ and $\mathcal{C} \cap \pa U=\emptyset$. The vector field $X$ can be extended (call it $\tilde{X}$) to the interior of $U$ with finitely many isolated zeros $\{p_1,p_2,...,p_n\}$. As defined in \cite[Definition 1.1.3]{Brasselet87}, the \textit{total index} of the vector field $X$ at $\tilde{\mathcal{C}}$ is the sum $ \ind_{\tilde{\mathcal{C}}}(X):=\sum_{i=1}^n \ind_{p_i}(\tilde{X})$. Since the sum is independent of the choice of the extension $\tilde{X}$ by Theorem \ref{thm: extended vector field -total index}, the definition of total index is well-defined. By a similar argument as in the isolated case, the total index of $X$ at any connected component $\tilde{\CCC}$ of the zero set $\CCC$ is given by the winding number of $X \circ \gamma$ around $0$, where $\gamma$ is any positively oriented smooth parametrisation of $\pa U$.  So, we have the following generalisation of Proposition \ref{prop: winding number is constant}:

\begin{corollary}\label{cor: winding number is constant, non-isolated}
    Let $X_t \in \mathfrak{X}^1(M)$ be a 1-parameter family of vector fields on $M$ such that the map $t \mapsto X_t \in \mathfrak{X}^1(M)$ is continuous at $t_0 \in \RR$ (in the $C^1$-topology) and $\mathcal{C}_t$ be the corresponding set of zeros.
    Let $\tilde{\mathcal{C}}_{t_0}$ be a finite union of connected components of $\mathcal{C}_{t_0} \subset M$. For each component, choose an isolating neighbourhood (a disc); let $U$ be their disjoint union. Define the total winding number $w(t)$ as the sum of the winding numbers of $X_t\circ\gamma_j$ around $0$ over all positively oriented boundary components $\gamma_j$ of $\partial U$. Then $w(t)$ is continuous at $t_0$, and hence, for all $t$ sufficiently close to $t_0$, $X_t\neq 0$ on $\partial U$ and the total index $\ind_{\mathcal{C}_t \cap U}(X_t)$ is locally constant. 
\end{corollary}

We apply the above %proposition and the lemma 
to eigenfunctions of the Schr\"{o}dinger operator $\Ht$ on a 1-parameter family of smooth Riemannian manifolds $M_t := (M, g_t), t\in \RR$. Since the eigenbranch satisfies $\varphi_t\to\varphi_0$ in $C^\infty$ (as recorded in the standing framework), the map $t\mapsto\nabla\varphi_t$ is in particular continuous in the $C^1$-topology, which is the regularity needed for the preceding results.
In what follows, we assume without any loss of generality that $t_0 = 0$ for convenience of notation.  
We now have the following proposition:
\begin{proposition} \label{prop: stability of critical points}
    Let $p$ be an isolated critical point of $\varphi_{0}$ with $\ind_p(\nabla \varphi_{0})$ positive (respectively, negative). For each neighbourhood $U$ of $p$, there exists $\epsilon>0$ such that if $|t|<\epsilon$ the perturbed eigenfunction $\varphi_t$ has a set of critical points in $U$ whose total index is positive (respectively, negative).
\end{proposition}
\begin{proof}
 Choose an isolating neighbourhood $V\Subset U$ of $p$ (a disc around $p$ so small that $\nabla\varphi_0\neq 0$ on $\overline{V}\setminus\{p\}$). Since $\ind_p(\nabla \varphi_0) \neq 0$, the winding number of the curve $\nabla \varphi_0 \circ \gamma$ around $0$ is non-zero, where $\gamma$ is a positively oriented parametrisation of $\partial V$. An application of Corollary \ref{cor: winding number is constant, non-isolated} to $X_t= \nabla \varphi_t$ implies that the winding number of $\nabla \varphi_t \circ \gamma$ around $0$ is a non-zero constant for $|t|\le\epsilon$. If $\nabla\varphi_t\neq 0$ on $\overline{V}$, then the map $\nabla\varphi_t/|\nabla\varphi_t|:\overline{V}\to S^1$ provides a continuous extension of the boundary map to $V$, forcing the winding number (= degree of the boundary map) to be zero, a contradiction. 
 Therefore $\varphi_t$ must have critical points in $V\subset U$ with non-zero total index and matching sign as $\ind_p(\nabla\varphi_0)$.
\end{proof}

Generally, it is possible to have a smooth vector field $X$ with an isolated zero $p$ such that $\ind_p(X) = 0$. However, for nodal critical points of eigenfunctions, the local structure forces the index to be strictly negative:
\begin{obsv}\label{obsv:2.4}
   At any nodal critical point $p$ of an eigenfunction $\varphi$ of $\Ht$ with vanishing order $k\ge 2$, $\ind_p(\nabla \varphi) = 1-k<0$. Indeed, by the Hartman--Wintner/Bers--Cheng expansion \cite{HartmanWintner1953,alessandrini1992}, in isothermal coordinates centred at $p$,
   \[
   \varphi(r,\theta)=ar^k\cos k(\theta-\theta_0)+O(r^{k+1}),\qquad a\neq 0.
   \]
   The gradient of the leading term $P_k=ar^k\cos k(\theta-\theta_0)$ is a homogeneous vector field of degree $k-1$ with an isolated zero at the origin whose winding number is $1-k$. Since $\nabla\varphi-\nabla P_k=O(r^{k-1+\alpha})$ for some $\alpha>0$, the winding number of $\nabla\varphi$ on a sufficiently small circle equals that of $\nabla P_k$, hence $\ind_p(\nabla\varphi)=1-k$.
\end{obsv}

\begin{remark}
At an isolated nondegenerate (Morse) saddle point, $\ind_p(\nabla\varphi)=-1$; at an isolated local maximum or minimum, $\ind_p(\nabla\varphi)=+1$. These are standard facts. However, for a general isolated saddle of a Schr\"{o}dinger eigenfunction (without the Morse nondegeneracy hypothesis), the index need not be negative.
\end{remark}

The last observation implicitly assumes that nodal critical points are isolated in the set of all critical points. This is justified by the following proposition (the latter part is contained in \cite{Cheng1976}).

\begin{proposition}\label{prop:nod_crt_isol}
    Any nodal critical point $p$ of an eigenfunction $\varphi$ of $\Ht$ is isolated in the set of critical points of $\varphi$. In particular, the number of nodal critical points is finite.
\end{proposition}
\begin{proof}
    Suppose $\varphi$ vanishes at $p$ to finite order $k\ge 2$. Choose isothermal coordinates centred at $p$, so that locally the equation for $\varphi$ takes the form
    \[
    -\Delta u + q(x)u = 0
    \]
    with $q$ smooth. By Chen \cite[Theorem~2.1, Corollary~2.2, Remark~2.3]{Chen1997}, applied in dimension $2$, there exist $\alpha>0$ and a nonzero homogeneous harmonic polynomial $P_k$ of degree $k$ such that
    \[
    \varphi(x)=P_k(x)+O(|x|^{k+\alpha}),
    \qquad
    \nabla \varphi(x)=\nabla P_k(x)+O(|x|^{k-1+\alpha}).
    \]
    Since $P_k$ is a nonzero homogeneous harmonic polynomial of degree $k$ in two variables, after a rotation we may write
    \[
    P_k(r,\theta)=ar^k\cos k(\theta-\theta_0),
    \qquad a\neq 0.
    \]
    Hence
    \[
    |\nabla P_k(r,\theta)|=|a|kr^{k-1}.
    \]
    Therefore, for $x\neq p$ sufficiently close to $p$, the error term is lower order and we get
    \[
    |\nabla \varphi(x)|\ge c|x-p|^{k-1}>0
    \]
    for some constant $c>0$. Thus $p$ is isolated in the critical set $\mathcal C(\varphi)$.

    The second part follows from compactness of $M$: the set of nodal critical points is discrete, hence closed and finite.
\end{proof}

We now prove that the set of nodal critical points $\mathcal{S}_t:=\mathcal{S}(\varphi_t)$ of $\varphi_t$ remains close to $\mathcal{S}(\varphi_0)$ under small perturbations. Throughout, all geodesic balls $B(p,r)$ are taken with respect to the fixed background metric $g_0$, and we choose $r$ below the injectivity radius of $(M,g_0)$.
    \begin{proposition}\label{prop: stability - general}
    The set of nodal critical points is stable under perturbation in the following sense: if $p_1, \dots, p_n$ are the nodal critical points of $\varphi_0$, then there exists $r_0>0$ such that for every $0<r\le r_0$ there exists $\delta>0$ (depending on $r$) with 
\begin{equation*}
    \mathcal{S}_t \subset \bigcup_{i=1}^n B(p_i,r) \subset M \quad \text{ whenever } |t|<\delta.
\end{equation*}
\end{proposition}
    \begin{proof}
   By Proposition~\ref{prop:nod_crt_isol}, each nodal critical point $p_i$ is isolated in the full critical set $\mathcal C(\varphi_0)$. Hence there exists $r_0>0$ such that for any $0<r\le r_0$ the $g_0$-geodesic balls $B(p_i,r)$ are pairwise disjoint, each punctured ball $\overline{B(p_i,r)}\setminus\{p_i\}$ contains no nodal critical points of $\varphi_0$, and $\partial B(p_i,r)\cap\mathcal{C}(\varphi_0)=\emptyset$. Fix such an $r$ and define $M':= M\setminus \bigcup_{i=1}^n B(p_i,r)$.
   
   On $M'$, the eigenfunction $\varphi_0$ has no nodal critical points, so
   \[
   F_0(x):=|\varphi_0(x)|+\|\nabla^{g_0}\varphi_0(x)\|_{g_0}>0\qquad\text{for all }x\in M'.
   \]
   Since $M'$ is compact, $F_0\ge 2m$ on $M'$ for some $m>0$. By $C^\infty$-convergence of $\varphi_t\to\varphi_0$ (and hence $C^1$-convergence of $\nabla^{g_0}\varphi_t\to\nabla^{g_0}\varphi_0$, where both gradients are computed with the fixed background metric $g_0$ as declared in the standing framework), there exists $\delta>0$ such that
   \[
   |\varphi_t(x)|+\|\nabla^{g_0}\varphi_t(x)\|_{g_0}\ge m>0\qquad\text{for all }x\in M',\ |t|<\delta.
   \]
   Therefore $\varphi_t$ and $\nabla^{g_0}\varphi_t$ cannot vanish simultaneously on $M'$, so all nodal critical points of $\varphi_t$ lie inside $\bigcup_{i=1}^n B(p_i,r)$.
\end{proof}
   \begin{remark}\label{remark: mod_prop_2.6}
     In the proof of Proposition \ref{prop: stability - general} we can make $\delta$ uniform under small change in the radius $r$: for $\rho>0$ small enough, $\mathcal{S}_0\cap N_\rho(r)=\emptyset$ where $N_\rho(r):=\bigcup_{i=1}^n\overline{B(p_i,r+\rho) \setminus B(p_i,r-\rho)}$, and then the same argument applied to $r-\rho$ gives $\mathcal{S}_t \subset \bigcup_{i=1}^n B(p_i,r-\rho)$ for all $|t|<\delta$.
   \end{remark}
For each nodal critical point $p_i$, by first choosing $r\le r_0$ so that $\partial B(p_i,r)\cap\mathcal C(\varphi_0)=\emptyset$ (guaranteed by Proposition~\ref{prop:nod_crt_isol}), and then applying Lemma~\ref{lem:lim_M_1}(ii), we may also ensure that $\partial B(p_i,r)$ contains no critical points of $\varphi_t$ at all for small $|t|$. Moving forward, the choice of a sufficiently small $r$ will be made with these conditions in mind for every $p_i$ simultaneously.

The following is an easy to prove (but fundamental) fact about $C^k$ convergence of eigenfunctions, $k \geq 1$ that was implicitly used in the proof of Proposition \ref{prop: stability - general}.
\begin{Lemma}\label{lem:lim_M_1}
Let $t_n\to 0$ and $x_n\in M$.
(i) If $x_n\in \mathcal N(\varphi_{t_n})$ for all $n$ and $x_n\to x$, then $x\in \mathcal N(\varphi_0)$.
(ii) If $x_n\in \mathcal C(\varphi_{t_n})$ for all $n$ and $x_n\to x$, then $x\in \mathcal C(\varphi_0)$.
\end{Lemma}
\begin{proof}
Part~(i): $\varphi_{t_n}(x_n)=0$ and $\varphi_{t_n}\to\varphi_0$ uniformly ($C^0$-convergence) give $\varphi_0(x)=\lim\varphi_{t_n}(x_n)=0$. Part~(ii): similarly, $d\varphi_{t_n}(x_n)=0$ and $C^1$-convergence of $\varphi_{t_n}\to\varphi_0$ give $d\varphi_0(x)=0$.
\end{proof}

For a proof in a more general setting, see \cite[Lemma 3.5]{MukherjeeSaha2022}. This has the following implication:
\begin{proposition}\label{prop:nodal_line_no_crossing}
     If $\varphi_0$ has no nodal critical point, then $\nu(\varphi_t) = \nu(\varphi_0)$ for sufficiently small $t$.
\end{proposition}

\begin{proof}
    If $\NNN(\varphi_0)=\emptyset$ (i.e.\ $\varphi_0$ has constant sign), then $m:=\min_M|\varphi_0|>0$ and $C^\infty$-convergence gives $|\varphi_t|\ge m/2>0$ for small $|t|$, hence $\nu(\varphi_t)=1=\nu(\varphi_0)$.
    
    Now assume $\NNN(\varphi_0)\neq\emptyset$. Since $\varphi_0$ has no nodal critical points, $0$ is a regular value of $\varphi_0$. Define $F:M\times(-\varepsilon,\varepsilon)\to\RR$ by $F(x,t):=\varphi_t(x)$. The map $F$ is smooth (by the standing assumption that the eigenbranch depends smoothly on $t$, not merely by convergence at $t=0$). For every $x\in\NNN(\varphi_0)$ the differential $d_x F(\cdot,0)=d\varphi_0(x)$ is nonzero. Since $\NNN(\varphi_0)$ is compact, there is a neighbourhood $U_0$ of $\NNN(\varphi_0)$ on which $|d\varphi_0|$ is bounded below. On the complementary compact set $K:=M\setminus U_0$ we have $|\varphi_0|\ge c>0$, hence by uniform convergence $|\varphi_t|\ge c/2$ on $K$ for small $|t|$. Therefore every zero of $\varphi_t$ lies in $U_0$, where $|d\varphi_t|>0$ by $C^1$-closeness; equivalently, $0$ is a regular value of $\varphi_t$ for all sufficiently small $|t|$. It follows that $\Sigma:=F^{-1}(0)\subset M\times(-\varepsilon,\varepsilon)$ is a smooth submanifold. The projection $\pi:\Sigma\to(-\varepsilon,\varepsilon)$ is a proper submersion: properness holds because $M$ is compact; submersivity holds because at each $(x,t)\in\Sigma$ the differential $d_x\varphi_t\neq 0$ implies that $T_{(x,t)}\Sigma$ projects surjectively onto the $t$-axis.
    
    By Ehresmann's fibration theorem, $\pi$ is a locally trivial fibre bundle. The trivialisation provides a smooth family of embeddings $e_t:\NNN(\varphi_0)\hookrightarrow M$ with $e_0=\mathrm{id}$ and $e_t(\NNN(\varphi_0))=\NNN(\varphi_t)$ (here we identify $\pi^{-1}(t)$ with $\NNN(\varphi_t)\subset M$ via projection to the first factor). For small $|t|$, $\NNN(\varphi_t)$ is therefore a compact embedded $1$-submanifold of the closed surface $M$, i.e.\ a finite disjoint union of embedded circles. By the isotopy extension theorem for compact submanifolds of smooth manifolds (see Hirsch \cite[Chapter 8, \S1, pp. 177 - 180]{Hirsch1976}), the smooth isotopy $e_t$ extends to a smooth ambient isotopy $E_t:M\to M$ with $E_0=\mathrm{id}_M$ and $E_t|_{\NNN(\varphi_0)}=e_t$. In particular, $M\setminus\NNN(\varphi_t)$ is homeomorphic to $M\setminus\NNN(\varphi_0)$ for all $|t|<\varepsilon$, and hence $\nu(\varphi_t)=\nu(\varphi_0)$.
    \end{proof}

\begin{remark}[Sequential version of the equality]\label{rem:sequential_noncritical}
The equality $\nu(\varphi_t)=\nu(\varphi_0)$ in Proposition~\ref{prop:nodal_line_no_crossing} remains valid in the purely sequential setting: for any sequence $u_n$ of eigenfunctions of $-\Delta_{g_{t_n}}+\Pot_{t_n}$ with $t_n\to 0$, $u_n\to u_0$ in $C^\infty(M)$, and $u_0$ an eigenfunction of $\Ht[0]$ with no nodal critical point, one has $\nu(u_n)=\nu(u_0)$ for $n\gg 1$. No smooth eigenbranch in $t$ is needed.

The proof above cannot be run verbatim in this setting: without joint smoothness of $F(x,t):=\varphi_t(x)$ in $(x,t)$, the zero set $\Sigma=F^{-1}(0)$ need not be a smooth submanifold over the $t$-axis, so Ehresmann's theorem does not apply. A direct argument nonetheless suffices, which we briefly sketch.

Assume $\NNN(u_0)\neq\emptyset$; the constant-sign case is handled exactly as in the proof above. Since $u_0$ has no nodal critical point, $\NNN(u_0)\cap\mathcal C(u_0)=\emptyset$, and both sets are compact, so $\mathrm{dist}(\NNN(u_0),\mathcal C(u_0))>0$. Write $\NNN(u_0)=C_1\sqcup\cdots\sqcup C_p$ as a disjoint union of embedded smooth circles, and fix
\[
  0<\delta<\tfrac12\min\Bigl\{\mathrm{dist}(C_i,C_j)\ (i\neq j),\ \ \mathrm{dist}(\NNN(u_0),\mathcal C(u_0))\Bigr\}.
\]
Let $U_j$ be the $\delta$-tubular neighbourhood of $C_j$; for $\delta$ below the injectivity radius, each $U_j$ is an embedded open annulus. Set $U:=\bigsqcup_j U_j$. The second constraint on $\delta$ ensures $\overline U\cap\mathcal C(u_0)=\emptyset$, hence $|du_0|\ge \eta>0$ on $\overline U$ for some $\eta$.

\smallskip
\emph{Step 1 (localisation).} $\NNN(u_n)\subset U$ for $n\gg 1$. If not, along a subsequence there exist $x_n\in\NNN(u_n)\setminus U$; by compactness of $M$ a further subsequence converges to some $x\in M\setminus U$, and Lemma~\ref{lem:lim_M_1}(i) gives $x\in\NNN(u_0)\setminus U$, a contradiction.

\smallskip
\emph{Step 2 (existence).} Each $U_j$ contains at least one component of $\NNN(u_n)$ for $n\gg 1$. Fix a short smooth arc $\gamma_j\subset U_j$ transverse to $C_j$, with endpoints $p_j^\pm$ satisfying $u_0(p_j^-)<0<u_0(p_j^+)$. By $C^0$-closeness, $u_n(p_j^-)<0<u_n(p_j^+)$ for $n\gg 1$, and the intermediate value theorem applied to $u_n|_{\gamma_j}$ produces a zero of $u_n$ on $\gamma_j\subset U_j$.

\smallskip
\emph{Step 3 (smooth 1-manifold structure).} By $C^1$-closeness on $\overline U$ we have $|du_n|\ge\eta/2$ on $\overline U$ for $n\gg 1$; in particular $\NNN(u_n)\cap\overline U$ is a smooth $1$-submanifold. Since $u_0\neq 0$ on $\partial U_j$ (by the first constraint on $\delta$), $C^0$-closeness gives $u_n\neq 0$ on $\partial U_j$ for $n\gg 1$. Hence every connected component of $\NNN(u_n)\cap U_j$ is a compact embedded circle lying in the interior of $U_j$.

\smallskip
\emph{Step 4 (uniqueness).} Each $U_j$ contains at most one component of $\NNN(u_n)$ for $n\gg 1$. Suppose for contradiction that $\NNN(u_n)\cap U_j$ has $\ge 2$ components. A disjoint union of $k\ge 2$ embedded circles in an open annulus divides it into at least $k+1$ connected components, but only at most two of these touch $\partial U_j$ (the two boundary circles of $U_j$ have definite signs, since $u_n\neq 0$ on $\partial U_j$). Hence at least one component $D$ of $U_j\setminus\NNN(u_n)$ is compactly contained in $U_j$, with $u_n$ of definite sign on $D$ and vanishing on $\partial D$. Then $D$ is a nodal domain of $u_n$ with $\mathrm{Area}_{g_{t_n}}(D)\le \mathrm{Area}_{g_{t_n}}(U_j)\le C_0\delta$, where $C_0$ depends only on the length of $C_j$ and on $C^2$-closeness of $g_{t_n}$ to $g_0$. By Lemma~\ref{lem:small-domain-Schro} applied to $u_n$ on $D$,
\[
  \lambda_{t_n}\ \ge\ \frac{C_{\mathrm{FK}}}{C_0\delta}-\sup_n\|\Pot_{t_n}\|_{L^\infty}.
\]
The right-hand side tends to $\infty$ as $\delta\to 0$, contradicting $\lambda_{t_n}\to\lambda_0$ for $\delta$ small.

\smallskip
Steps 1-4 give a bijection between components of $\NNN(u_0)$ and components of $\NNN(u_n)$, and the sign pattern of $u_n$ on $M\setminus\NNN(u_n)$ matches that of $u_0$ on $M\setminus\NNN(u_0)$ (by $C^0$-closeness on compact subsets of $M\setminus U$). Hence $\nu(u_n)=\nu(u_0)$ for $n\gg 1$. The price of dropping the smooth-branch assumption is that this argument yields only the count equality, not an ambient isotopy of $M$ carrying $\NNN(u_0)$ to $\NNN(u_n)$.
\end{remark}

\section{Global nodal domain monotonicity: proofs of Theorems \ref{thm:loc_nod_est} and \ref{thm:main_result_A}}\label{Sec: Perturbation and critical points}
 
Before beginning with the proofs, we introduce the following crucial lemma which shall be used in multiple instances throughout the remainder of the article. 

\begin{Lemma}\label{lem:nod_lin_disc_invariant}
    Let $p$ be an isolated nodal critical point of $\varphi_0$ with order of vanishing $k$. Then there exists $r_0>0$ such that for \emph{every} $r\in(0,r_0)$, the circle $\partial B(p,r)$ meets $\mathcal N(\varphi_0)$ transversely in exactly $2k$ points, and for all sufficiently small $|t|$ (depending on $r$),
    \[
      \#\bigl(\mathcal N(\varphi_t)\cap\partial B(p,r)\bigr)\ =\ 2k.
    \]
\end{Lemma}

\begin{proof}
Choose isothermal coordinates centred at $p$ and define $B(p,r)$ to be the coordinate disc of radius $r$ in these coordinates (for $r$ below the injectivity radius, this is comparable to the $g_0$-geodesic ball). By the Bers--Cheng expansion,
\[
\varphi_0(r,\theta)=ar^k\cos k(\theta-\theta_0)+O(r^{k+1}),\qquad a\neq 0.
\]
For any fixed $r\in(0,r_0)$, the boundary trace $f_0(\theta):=\varphi_0(r,\theta)$ satisfies $r^{-k}f_0(\theta)=a\cos k(\theta-\theta_0)+O(r)$. For $r$ small, $r^{-k}f_0$ is a $C^1$-small perturbation of $a\cos k(\theta-\theta_0)$, hence has exactly $2k$ simple zeros on $[0,2\pi)$. In particular, $\partial B(p,r)$ meets $\mathcal N(\varphi_0)$ transversely at exactly $2k$ points.

Since $\varphi_t\to\varphi_0$ in $C^\infty$, the restriction $f_t:=\varphi_t|_{\partial B(p,r)}$ converges to $f_0$ in $C^1(\partial B(p,r))$. At each simple zero $\theta_j$ of $f_0$, the implicit function theorem gives a unique simple zero $\theta_j(t)$ of $f_t$ near $\theta_j$ for small $|t|$. Since $|f_0|$ is bounded below on the complement of small intervals around the zeros, $f_t$ has no other zeros for $|t|$ small. Hence $\#(\mathcal N(\varphi_t)\cap\partial B(p,r))=2k$.
\end{proof}

\begin{remark}\label{Remark: 3.2.}
For the remainder of the section, we fix once and for all a radius $r\in(0,r_0)$ satisfying additionally $\partial B(p,r)\cap\mathcal C(\varphi_t)=\emptyset$ for small $|t|$. This can be arranged: first choose $r$ so that $\partial B(p,r)\cap\mathcal C(\varphi_0)=\emptyset$ (possible since $p$ is isolated in $\mathcal C(\varphi_0)$ by Proposition~\ref{prop:nod_crt_isol}), then apply Lemma~\ref{lem:lim_M_1}(ii) by contradiction: if $x_n\in\partial B(p,r)\cap\mathcal C(\varphi_{t_n})$ with $t_n\to 0$, then by compactness $x_n\to x\in\partial B(p,r)\cap\mathcal C(\varphi_0)$, contradicting the choice of $r$.
\end{remark}

The following auxiliary lemma isolates the Jordan-curve-plus-small-domain argument that is used repeatedly below. It simultaneously rules out two a priori distinct configurations: smooth $S^1$ components of $\mathcal N(\varphi_t)$ disjoint from $\partial B(p,r)$ (used in the next paragraph), and cycles in the nodal graph $G_t$ that traverse interior vertices (used in Lemma~\ref{lem:forest}).

\begin{Lemma}[Closed-loop exclusion]\label{lem:no-closed-loop}
For $r$ sufficiently small and $|t|$ sufficiently small, the set $\mathcal N(\varphi_t)\cap\overline{B(p,r)}$ contains no embedded topological circle (smooth or piecewise smooth) that is compactly contained in $B(p,r)$.
\end{Lemma}
\begin{proof}
Suppose $\Gamma\subset B(p,r)$ is such an embedded circle. Since $\overline{B(p,r)}$ is a topological disc, by the Jordan curve theorem $\Gamma$ bounds a simply connected region $D\subset B(p,r)$. As $\varphi_t$ changes sign across the smooth portion of its nodal set and the finitely many nodal critical points on $\Gamma$ form a negligible set, $\varphi_t$ is of definite sign on a nonempty open subset of $D$; extending this subset to a maximal sign-definite component gives a nodal domain of $\varphi_t$ contained in $D\subset B(p,r)$. By Lemma~\ref{lem:small-domain-Schro}, applied uniformly for small $|t|$ (since $\lambda_t\to\lambda_0$ and $\|\Pot_t\|_{C^0}$ is bounded in the standing framework), this is impossible once $r$ is small enough.
\end{proof}

\paragraph{The full nodal graph.}
We now impose a graph structure on $\mathcal N(\varphi_t)\cap\overline{B(p,r)}$. Define the \emph{full nodal graph} $G_t$ as the embedded planar graph in $\overline{B(p,r)}$ with:
\begin{itemize}
\item \emph{interior vertices}: the nodal critical points $\mathcal S_t\cap B(p,r)$, each of degree $2\tilde k$ where $\tilde k\ge 2$ is its corresponding vanishing order;
\item \emph{boundary vertices}: the $2k$ points $\mathcal N(\varphi_t)\cap\partial B(p,r)$, each of degree $1$ (because the corresponding zeros of $f_t=\varphi_t|_{\partial B(p,r)}$ are simple, by Lemma~\ref{lem:nod_lin_disc_invariant});
\item \emph{edges}: the closures of the connected components of $(\mathcal N(\varphi_t)\cap\overline{B(p,r)})\setminus\{\text{vertices}\}$.
\end{itemize}
By Lemma~\ref{lem:nod_lin_disc_invariant}, the number of boundary vertices is $2k$ for all small $|t|$. By Lemma~\ref{lem:no-closed-loop}, no connected component of $\mathcal N(\varphi_t)\cap\overline{B(p,r)}$ is a smooth embedded circle contained in $B(p,r)$; every connected component therefore either meets $\partial B(p,r)$ (and is thus incident to at least one boundary vertex) or contains at least one nodal critical point (and is thus incident to at least one interior vertex). In either case it is captured by $G_t$, so $G_t$ exhausts the entire nodal set $\mathcal N(\varphi_t)\cap\overline{B(p,r)}$. Since adjoining the boundary endpoints to the open-ball portion does not change the number of connected components, the number of components $C_t$ of $G_t$ equals $\#\pi_0(\mathcal N(\varphi_t)\cap B(p,r))$ as used in the statement of Theorem~\ref{thm:loc_nod_est}.

\begin{Lemma}[Acyclicity]\label{lem:forest}
For $r$ sufficiently small and $|t|$ sufficiently small, the full nodal graph $G_t$ is a forest (i.e.\ it contains no cycle).
\end{Lemma}

\begin{remark}
We stress that this is a genuinely separate statement from the preceding paragraph. The paragraph rules out smooth circle components of $\mathcal N(\varphi_t)$ (i.e.\ connected components of $\mathcal N(\varphi_t)$ that are embedded $S^1$'s); Lemma~\ref{lem:forest} rules out \emph{cycles in $G_t$}, which may traverse interior vertices and are in general not connected components of $\mathcal N(\varphi_t)$ (e.g.\ two interior vertices joined by two arcs, with additional arcs reaching $\partial B(p,r)$). Both facts are needed for the identity~\eqref{eqn: index_connected_component_relation}: the former for $G_t$ to exhaust the nodal set, the latter for the tree-counting relation $V-E=C_t$. Lemma~\ref{lem:no-closed-loop} provides a single Jordan-plus-small-domain mechanism underlying both.
\end{remark}

\begin{proof}[Proof of Lemma~\ref{lem:forest}]
Since boundary vertices have degree $1$, no cycle in $G_t$ can pass through a boundary vertex; hence any cycle is an embedded topological circle contained in the interior $B(p,r)$. Such a circle is ruled out by Lemma~\ref{lem:no-closed-loop}.
\end{proof}

We now record the key combinatorial identity. Let $C_t$ denote the number of connected components of $G_t$. For each component $G_t^i$ ($i=1,\dots,C_t$), let $m_i$ be the number of interior vertices, $2b_i$ the number of boundary vertices (this is even because all interior vertices have even degree, so the number of odd-degree vertices, i.e.\ the boundary leaves, is even by the handshake lemma), and $\tilde k_{i,l}$ ($l=1,\dots,m_i$) the vanishing orders of the interior vertices. Since $G_t^i$ is a tree with $m_i+2b_i$ vertices, it has $m_i+2b_i-1$ edges. By the handshake lemma, the sum of vertex degrees is twice the edge count:
\[
\sum_{l=1}^{m_i}2\tilde k_{i,l}+2b_i\ =\ 2(m_i+2b_i-1),
\]
which simplifies to
\begin{equation}\label{eqn:forest_identity_component}
\sum_{l=1}^{m_i}(\tilde k_{i,l}-1)\ =\ b_i-1.
\end{equation}
Summing over $i=1,\dots,C_t$ and using $\sum_i 2b_i=2k$ (Lemma~\ref{lem:nod_lin_disc_invariant}):
\begin{equation}\label{eqn: index_connected_component_relation}
\sum_{q\in\mathcal S_t\cap B(p,r)}(\operatorname{ord}_q(\varphi_t)-1)\ =\ \sum_{i=1}^{C_t}\sum_{l=1}^{m_i}(\tilde k_{i,l}-1)\ =\ \sum_{i=1}^{C_t}(b_i-1)\ =\ k-C_t.
\end{equation}
(Here $\operatorname{ord}_q(\varphi_t)$ is the vanishing order of $\varphi_t$ at $q$, which equals $\tilde k$ when the nodal graph vertex at $q$ has degree $2\tilde k$; equivalently, $\operatorname{ord}_q(\varphi_t)-1=|\ind_q(\nabla\varphi_t)|$ by Observation~\ref{obsv:2.4}.)

\medskip

\begin{figure}[ht]
    \centering
    \includegraphics[scale=0.23]{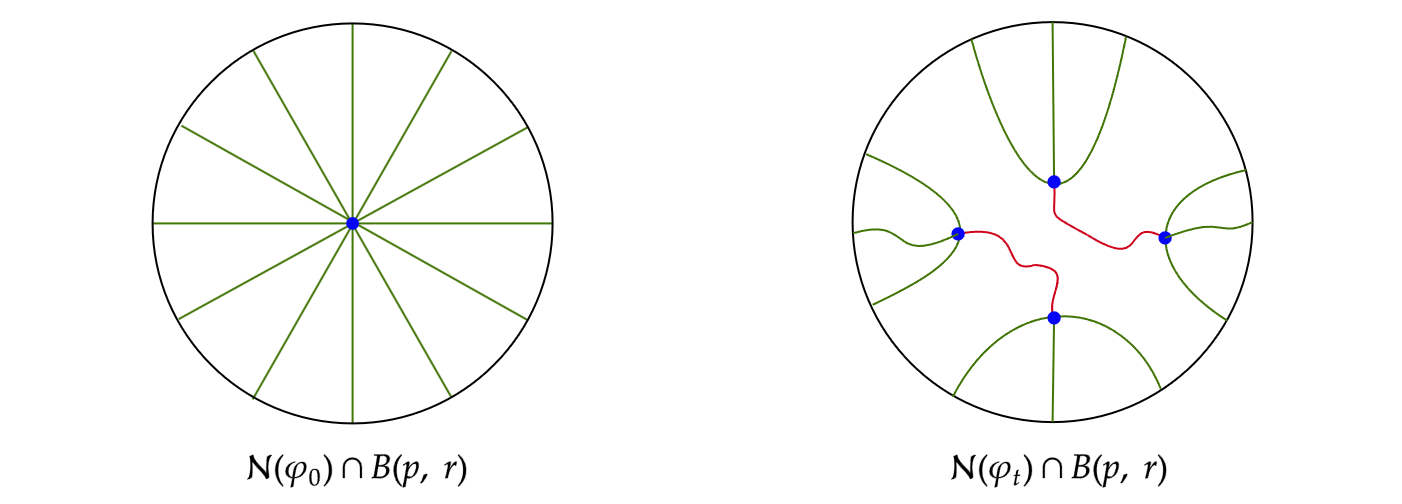}
    \caption{The full nodal graph $G_t$: interior vertices (blue), interior-interior edges (red), and interior-boundary or boundary-boundary edges (green). Boundary vertices (not shown) lie on $\partial B(p,r)$.}
    \label{fig:nodal_graph_ex}
\end{figure}

\subsection{Proof of Theorem \ref{thm:loc_nod_est}(a)}
Given $r$ small enough, there exists $\varepsilon:=\varepsilon(r)>0$ such that for $|t|<\varepsilon$ the number of nodal critical points of $\varphi_t$ inside $B(p,r)$ does not exceed $k-1$.

\smallskip\noindent\emph{Proof.} By Proposition~\ref{prop:nod_crt_isol}, $\mathcal S_t\cap B(p,r)$ is finite; let $m:=|\mathcal S_t\cap B(p,r)|$. By \eqref{eqn: index_connected_component_relation}, $\sum_{q}(\operatorname{ord}_q(\varphi_t)-1)=k-C_t\le k-1$ (since $C_t\ge 1$: $G_t$ has at least $2k\ge 2$ boundary vertices, hence at least one component). Since each summand is $\ge 1$, we get $m\le k-1$. \qed

\subsection{Proof of Theorem \ref{thm:loc_nod_est}(b)}
Given $r$ small enough, there exists $\varepsilon:=\varepsilon(r)>0$ such that for $|t|<\varepsilon$, the order of vanishing of any nodal critical point of $\varphi_t$ within $B(p,r)$ cannot be greater than $k$.

\smallskip\noindent\emph{Proof.} If some $q\in\mathcal S_t\cap B(p,r)$ has vanishing order $l>k$, then $\operatorname{ord}_q(\varphi_t)-1=l-1\ge k$. But \eqref{eqn: index_connected_component_relation} gives $\sum_{q'}(\operatorname{ord}_{q'}(\varphi_t)-1)=k-C_t\le k-1$, a contradiction. \qed

\subsection{Proof of Theorem \ref{thm:loc_nod_est}(c)}\label{subsec:part_c}

We first record a general combinatorial fact about forests in a disc, then apply it to the local nodal graph.

\begin{proposition}[Sector count by Euler's formula]\label{thm:loc_nod_est_alt_version}
Let $F\subset\overline{B(p,r)}$ be a finite embedded forest with $2k$ leaves on $\partial B(p,r)$ (each of degree $1$), and assume that every connected component of $F$ meets $\partial B(p,r)$. Let $C_F$ denote the number of connected components of $F$. Then the number $S(F)$ of connected components of $B(p,r)\setminus F$ is
\[
S(F)\ =\ 2k-C_F+1.
\]
We call these components the \emph{sectors} of $F$.
\end{proposition}
\begin{proof}
Let $V$ and $E$ denote the numbers of vertices and edges of $F$ (counting boundary leaves as vertices). Since every connected component of a forest is a tree, and each tree $T_i$ satisfies $E_i=V_i-1$, summing over the $C_F$ components gives
\[
E=V-C_F.
\]
Now form the \emph{augmented graph} $\widehat F$ by adding the $2k$ arcs of $\partial B(p,r)$ between consecutive boundary leaves as new edges. Because every connected component of $F$ meets the boundary, these boundary arcs connect all components of $F$, so $\widehat F$ is connected. The added arcs lie entirely on $\partial B(p,r)$, so they do not subdivide any interior component of $B(p,r)\setminus F$; they merely complete the boundary cycle. Hence the complementary components of $\widehat F$ in the closed disc $\overline{B(p,r)}$ are exactly the sectors of $F$.

Applying Euler's formula to the connected planar graph $\widehat F$, where the boundary circle $\partial B(p,r)$ is built into the graph (equivalently, gluing a second disc along $\partial B(p,r)$ passes to a sphere with exactly one additional exterior face, turning the spherical formula $V-E+F=2$ into its disc counterpart below), we get
\[
V-(E+2k)+S(F)=1.
\]
Substituting $E=V-C_F$ yields
\[
S(F)=1+(E+2k)-V=1+(V-C_F+2k)-V=2k-C_F+1.\qedhere
\]
\end{proof}

We can now prove Theorem~\ref{thm:loc_nod_est}(c). Recall that $F_t(p)$ denotes the number of connected components of $(M\setminus\NNN(\varphi_t))\cap B(p,r)$, i.e.\ the number of ``local nodal domains'' of $\varphi_t$ in $B(p,r)$. At $t=0$, the Hartman--Wintner/Alessandrini local structure (Section~\ref{Sec: Preliminary results}) gives that the nodal set in $B(p,r)$ is a single $2k$-star (one connected component with $2k$ arcs), so $F_0(p)=2k$.

\smallskip\noindent\emph{Proof of Theorem~\ref{thm:loc_nod_est}(c).}
For $r$ small enough, the full nodal graph $G_t=\NNN(\varphi_t)\cap\overline{B(p,r)}$ is a finite forest with $2k$ boundary leaves (Proposition~\ref{prop:nod_crt_isol}, Lemma~\ref{lem:nod_lin_disc_invariant}, and Lemma~\ref{lem:forest}). Every connected component of $G_t$ meets $\partial B(p,r)$: indeed, if some component did not, then as a finite tree it would have a leaf, but the only degree-$1$ vertices of $G_t$ are the boundary leaves. Let $C_t$ denote the number of connected components of $G_t$. By Proposition~\ref{thm:loc_nod_est_alt_version},
\[
S(G_t)\ =\ 2k-C_t+1.
\]
Since $G_t\cap B(p,r)=\NNN(\varphi_t)\cap B(p,r)$, we have $B(p,r)\setminus G_t=B(p,r)\setminus\NNN(\varphi_t)$, so by definition $F_t(p)=S(G_t)$. Therefore
\[
F_t(p)\ =\ 2k-C_t+1\ =\ F_0(p)-C_t+1.
\]
In particular, $F_t(p)\le F_0(p)$ since $C_t\ge 1$.

Moreover, every connected component of $B(p,r)\setminus\NNN(\varphi_t)$ is incident to at least one boundary arc of $\partial B(p,r)\setminus\NNN(\varphi_t)$. Indeed, if a component $\Omega$ were not incident to any boundary arc, then its only possible contact with $\partial B(p,r)$ would be at nodal points (zeros of $\varphi_t$). Hence $\varphi_t$ has constant sign on $\Omega$ and vanishes on $\partial\Omega\cap\partial B(p,r)$, so $\Omega$ is a nodal domain of $\varphi_t$ contained in $\overline{B(p,r)}$ with diameter $\le 2r$, contradicting Lemma~\ref{lem:small-domain-Schro} for $r$ small enough. \qed

\subsection{Proof of Theorem \ref{thm:main_result_A}} \label{Sec: Proofs of main theorem 2}

Let $p_1,\dots,p_n$ be the nodal critical points of $\varphi_0$, with vanishing orders $k_1,\dots,k_n\ge 2$. Choose $r>0$ small enough that the balls $B_i:=B(p_i,r)$ are pairwise disjoint, $r$ satisfies the conclusions of Proposition~\ref{prop: stability - general}, Lemma~\ref{lem:nod_lin_disc_invariant}, and Lemma~\ref{lem:forest}, and define
\[
M':=M\setminus \bigcup_{i=1}^n B_i.
\]
The surface $M'$ is a compact manifold with boundary $\partial M'=\bigcup_i\partial B_i$.

\begin{Lemma}[Nodal stability on the regular region]\label{Lem: nodal_set_invariance_M'}
For small $|t|$, there exists a diffeomorphism $H_t:M'\to M'$ carrying each boundary circle $\partial B_i$ to itself setwise and satisfying
\[
H_t\bigl(\NNN(\varphi_0)\cap M'\bigr)=\NNN(\varphi_t)\cap M'.
\]
Moreover, on each $\partial B_i$ the $2k_i$ zeros of $\varphi_t$ continue uniquely from those of $\varphi_0$, giving a bijection between the boundary arcs of $\partial B_i\setminus\NNN(\varphi_0)$ and those of $\partial B_i\setminus\NNN(\varphi_t)$. Under this identification:
\begin{enumerate}
\item[(i)] the connected components of $M'\setminus \NNN(\varphi_t)$ are in bijection with those of $M'\setminus \NNN(\varphi_0)$;
\item[(ii)] the sign of $\varphi_t$ on each boundary arc agrees with the sign of $\varphi_0$ on the corresponding arc;
\item[(iii)] if $\Omega_0$ is a component of $M'\setminus\NNN(\varphi_0)$ and $\Omega_t$ is the corresponding component of $M'\setminus\NNN(\varphi_t)$ under the bijection of~(i), then for every boundary arc $A_0$ of $\partial B_i\setminus\NNN(\varphi_0)$, $\Omega_0$ meets $A_0$ if and only if $\Omega_t$ meets the corresponding boundary arc $A_t$ of $\partial B_i\setminus\NNN(\varphi_t)$.
\end{enumerate}

\end{Lemma}

\begin{proof}
By Proposition~\ref{prop: stability - general}, after shrinking $|t|$ if necessary, every nodal critical point of $\varphi_t$ lies in $\bigcup_i B_i$. Hence $0$ is a regular value of $\varphi_t$ on the interior of $M'$.

For each zero $q_j\in\NNN(\varphi_0)\cap\partial B_i$, Lemma~\ref{lem:nod_lin_disc_invariant} gives $\partial_s\varphi_0(q_j)\neq 0$ in a boundary coordinate $s$ on $\partial B_i$. Since $\varphi_t|_{\partial B_i}\to\varphi_0|_{\partial B_i}$ in $C^1$, the implicit function theorem yields a unique nearby zero $q_j(t)\in\partial B_i$ of $\varphi_t$, depending smoothly on $t$, with $\partial_s\varphi_t(q_j(t))\neq 0$. Away from the zeros, $|\varphi_0|$ is bounded below on $\partial B_i$, so $C^0$-convergence gives $\varphi_t\neq 0$ there; hence these are the only boundary zeros. This yields the bijection of boundary arcs. Moreover, $\NNN(\varphi_t)$ meets $\partial B_i$ transversely at each boundary zero, so
\[
N_t:=\NNN(\varphi_t)\cap M'
\]
is a compact properly embedded $1$-submanifold of $M'$, transverse to $\partial M'$, for all small $|t|$.

\emph{Isotopy of the nodal set.} Define $F:M'\times(-\varepsilon,\varepsilon)\to\RR$ by $F(x,t):=\varphi_t(x)$. Since $d_x\varphi_t\neq 0$ at every point of $N_t$ (by regularity in the interior and transversality at $\partial M'$), $0$ is a regular value of the total map $F$, so $\Sigma:=F^{-1}(0)$ is a smooth submanifold of $M'\times(-\varepsilon,\varepsilon)$ with $\partial\Sigma\subset\partial M'\times(-\varepsilon,\varepsilon)$. Let $\pi:\Sigma\to(-\varepsilon,\varepsilon)$ denote the restriction to $\Sigma$ of the second-factor projection $M'\times(-\varepsilon,\varepsilon)\to(-\varepsilon,\varepsilon)$, $(x,t)\mapsto t$; its fibre over $t$ is $N_t\times\{t\}$, i.e.\ the nodal set at parameter $t$. The map $\pi$ is a proper submersion (properness by compactness; submersivity because $d_x\varphi_t\neq 0$ implies $T_{(x,t)}\Sigma$ surjects onto the $t$-axis), and its restriction to $\partial\Sigma$ is also a submersion (by transversality to $\partial M'$).

The boundary portion $\partial\Sigma$ consists of the curves $t\mapsto(q_j(t),t)$ given by the IFT continuation of boundary zeros. By the relative Ehresmann fibration theorem (see e.g.\ \cite[Section~8.1.2]{BroeckerJaenich}), since both $\pi|_\Sigma$ and $\pi|_{\partial\Sigma}$ are proper submersions, there exists a trivialisation of $\Sigma$ over $(-\varepsilon,\varepsilon)$ that \emph{extends} the trivialisation of $\partial\Sigma$ determined by the IFT continuation. In other words, there is a smooth family of embeddings $e_t:N_0\hookrightarrow M'$ with $e_t(N_0)=N_t$, $e_0=\mathrm{id}$, and $e_t(q_j)=q_j(t)$ for every boundary zero $q_j$.

Applying isotopy extension for properly embedded compact submanifolds (see e.g.\ \cite[Theorem~1.3]{Hirsch1976}) to the isotopy $e_t$ yields a diffeomorphism $H_t:M'\to M'$ carrying each boundary component $\partial B_i$ to itself setwise, satisfying $H_t(N_0)=N_t$, and agreeing with $e_t$ on $N_0\cap\partial M'$. In particular, $H_t(q_j)=q_j(t)$ for every boundary zero. This proves (i).

\emph{Proof of \textnormal{(ii)}.} Fix a boundary arc $A_0$ of $\partial B_i\setminus\NNN(\varphi_0)$, bounded by consecutive zeros $q_j,q_{j+1}$. The corresponding arc $A_t$ for $\varphi_t$ is bounded by the continued zeros $q_j(t),q_{j+1}(t)$. Choose any point $p_0\in A_0$; since $\varphi_0(p_0)\neq 0$ and $q_j(t)\to q_j$, $q_{j+1}(t)\to q_{j+1}$ smoothly, for small $|t|$ the point $p_0$ lies in $A_t$. By $C^0$-convergence, $\varphi_t(p_0)$ has the same sign as $\varphi_0(p_0)$. Since $\varphi_t$ is nonzero on $A_t$, the sign of $\varphi_t$ on $A_t$ equals that of $\varphi_0$ on $A_0$.

\emph{Proof of \textnormal{(iii)}.} Let $\Omega_0$ be a component of $M'\setminus N_0$ and $\Omega_t:=H_t(\Omega_0)$ the corresponding component of $M'\setminus N_t$. Since $H_t$ carries $N_0$ to $N_t$ and sends each boundary zero $q_j$ to $q_j(t)$, it carries each boundary arc $A_0$ (bounded by $q_j,q_{j+1}$) into the corresponding arc $A_t$ (bounded by $q_j(t),q_{j+1}(t)$). Therefore $\Omega_t$ meets $A_t$ if and only if $\Omega_0$ meets $A_0$, and the incidence pattern is preserved.
\end{proof}

\paragraph{The dual incidence graph.}
For a fixed $t$, define the \emph{dual incidence graph} $\mathcal G_t$ as the bipartite graph whose vertices are:
\begin{itemize}
\item the connected components of $M'\setminus\NNN(\varphi_t)$ (\emph{exterior vertices});
\item for each ball $B_i$, the connected components of $B_i\setminus\NNN(\varphi_t)$ (\emph{sector vertices}).
\end{itemize}
An exterior vertex $E$ and a sector vertex $S\subset B_i$ are joined by an edge whenever they are adjacent along a boundary arc of $\partial B_i\setminus\NNN(\varphi_t)$.

We claim that the connected components of $\mathcal G_t$ are in bijection with the nodal domains of $\varphi_t$ on $M$. First, every component of $B_i\setminus\NNN(\varphi_t)$ is incident to at least one boundary arc of $\partial B_i\setminus\NNN(\varphi_t)$ (by the last paragraph of the proof of Theorem~\ref{thm:loc_nod_est}(c)). Hence every nodal domain $D$ of $\varphi_t$ intersects $M'$ and the balls $B_i$ in a union of exterior components and sectors; these pieces are pairwise connected through boundary arcs, so the corresponding vertices form a connected subgraph of $\mathcal G_t$. Conversely, let $\Gamma$ be a connected component of $\mathcal G_t$, and let $W$ be the union of the corresponding exterior components and sectors. Since $\Gamma$ is connected, any two pieces in $W$ are linked by a chain of edges (boundary arcs); across each such arc $\varphi_t$ is continuous and nonzero, so the sign is constant throughout $W$. Therefore $W$ is a connected open set on which $\varphi_t$ has constant sign, hence $W$ is contained in a single nodal domain. Since every exterior component and every sector belongs to exactly one nodal domain (by sign-definiteness), these two assignments are inverse, establishing the bijection.

\begin{theorem}\label{thm:perturb_nodal_domain_reduction-saikat}
For small $|t|$,
\[
\nu(\varphi_t)\ \le\ \nu(\varphi_0).
\]
\end{theorem}

\begin{proof}
By Lemma~\ref{Lem: nodal_set_invariance_M'}, the exterior vertices of $\mathcal G_t$ are naturally identified with those of $\mathcal G_0$ (via the diffeomorphism $H_t$), and under this identification the boundary-arc incidences of exterior components are preserved in the sense of part~(iii): an exterior component $\Omega_t$ of $M'\setminus\NNN(\varphi_t)$ meets a boundary arc $A_t$ of $\partial B_i\setminus\NNN(\varphi_t)$ if and only if the corresponding $\Omega_0$ meets the corresponding $A_0$.

Fix a ball $B_i$ and let $k_i$ be the vanishing order of $\varphi_0$ at $p_i$. Let $C_t^{(i)}$ denote the number of connected components of the local nodal graph $G_t$ in $B_i$. For $t=0$, the nodal set in $B_i$ is the standard $2k_i$-star, so $B_i\setminus\NNN(\varphi_0)$ has exactly $2k_i$ sectors, one incident to each boundary arc. For small $t$, Theorem~\ref{thm:loc_nod_est}(c) gives $\#\{\text{sectors in }B_i\}=2k_i-C_t^{(i)}+1\le 2k_i$. Hence the sectors of $\varphi_t$ define a partition of the $2k_i$ boundary arcs that is coarser than the discrete partition at $t=0$. Moreover, each partition block is sign-homogeneous: all arcs in a block are incident to the same sector, on which $\varphi_t$ has constant sign, and by part~(ii) of Lemma~\ref{Lem: nodal_set_invariance_M'} these arcs had the same sign at $t=0$ as well.

Therefore $\mathcal G_t$ is obtained from $\mathcal G_0$ by identifying, within each ball $B_i$, those sector vertices of $\mathcal G_0$ whose corresponding boundary arcs lie in the same sector of $\varphi_t$. This is a quotient of the bipartite graph $\mathcal G_0$ (exterior vertices are untouched; sector vertices are merged according to the partition). Quotienting by vertex identifications cannot increase the number of connected components. Since connected components of $\mathcal G_t$ and $\mathcal G_0$ are in bijection with nodal domains, we conclude
\[
\nu(\varphi_t)=\#\pi_0(\mathcal G_t)\le \#\pi_0(\mathcal G_0)=\nu(\varphi_0).
\]
\end{proof}

This proves the first assertion of Theorem~\ref{thm:main_result_A}; the final assertion (equality when $\varphi_0$ has no nodal critical points) is Proposition~\ref{prop:nodal_line_no_crossing}.

\begin{remark}[Sequential version of the inequality]\label{rem:sequential_USC}
The inequality $\nu(\varphi_t)\le\nu(\varphi_0)$ in Theorem~\ref{thm:main_result_A} holds more generally for any sequence of eigenfunctions $u_n$ of $-\Delta_{g_{t_n}}+\Pot_{t_n}$ with $t_n\to 0$ and $u_n\to u_0$ in $C^\infty(M)$, where $u_0$ is an eigenfunction of $\Ht[0]$. We sketch why. Choose balls $B_i$ around the nodal critical points of $u_0$ and set $M':=M\setminus\bigcup_i B_i$ as before. Three ingredients are needed for the dual-graph quotient argument:
\begin{enumerate}
\item[(a)] \emph{Local inputs.} The boundary intersection counts, forest structure, and sector counts in each $B_i$ depend only on $C^\infty$-convergence of $u_n\to u_0$ on fixed compact subsets. These go through unchanged.
\item[(b)] \emph{Exterior stability.} On $M'$, $0$ is a regular value of $u_0$. By $C^1$-closeness, $0$ is a regular value of $u_n$ on $M'$ for $n\gg 1$. Proposition~\ref{prop:nodal_line_no_crossing} (applied on $M'$ to the pair $u_0,u_n$, using the isotopy extension theorem for compact submanifolds of a fixed manifold with boundary) gives the same number of exterior components and the same boundary-arc incidence.
\item[(c)] \emph{Sign preservation.} Boundary arcs on each $\partial B_i$ carry the same sign for $u_n$ as for $u_0$, by $C^0$-closeness.
\end{enumerate}
With (a)-(c), the proof of Theorem~\ref{thm:perturb_nodal_domain_reduction-saikat} goes through verbatim, giving $\nu(u_n)\le\nu(u_0)$ for $n\gg 1$. The smooth family structure $F(x,t)=\varphi_t(x)$ and Ehresmann's theorem (as a one-parameter fibration) are used only for the equality clause (Proposition~\ref{prop:nodal_line_no_crossing} in the smooth-branch setting), not for the inequality.
\end{remark}

\begin{remark}
The theorem above, together with the local sector-count formula, also yields a quantitative bound on the decrement. For each ball $B_i$, define the \emph{local decrement} $\Delta_i$ as the drop in the number of connected components of the dual graph caused by the coarsening at $B_i$ (so if the coarsening merges $m$ components into one, the decrement is $m-1$). The sector-partition argument (cf.\ Lemma~\ref{lem:sector_partition} below) gives $\Delta_i\le C_t^{(i)}-1$.
\end{remark}

\begin{Lemma}[Sector partition and maximal decrement]\label{lem:sector_partition}
Fix $p\in\mathcal S_0$ with vanishing order $k$, and let $B:=B(p,r)$ be as above. For small $|t|$, the disc boundary $\partial B$ meets $\NNN(\varphi_t)$ in exactly $2k$ points, and the nodal set inside $B$ has $C_t\ge1$ connected components and
\[
k' = 2k - C_t + 1
\]
sectors (Proposition~\ref{thm:loc_nod_est_alt_version}). The sectors $\{S_1,\dots,S_{k'}\}$ induce a partition $\{E_1,\dots,E_{k'}\}$ of the $2k$ boundary arcs $\{I_1,\dots,I_{2k}\}$, where $E_\alpha$ is the set of boundary arcs incident to $S_\alpha$.

Each sector $S_\alpha$ merges all exterior components meeting arcs in $E_\alpha$ into a single connected piece through $S_\alpha$. Since the number of distinct exterior components meeting $E_\alpha$ is at most $|E_\alpha|$, the drop in component count from this sector is at most $|E_\alpha|-1$. Summing over sectors:
\[
\sum_{\alpha=1}^{k'} (|E_\alpha|-1)\ =\ 2k - k' \ =\ C_t - 1.
\]
\end{Lemma}

\begin{proof}
Each sector $S_\alpha$ is connected, so after gluing, all exterior components meeting arcs in $E_\alpha$ become connected through $S_\alpha$, reducing their count by at most $|E_\alpha|-1$. Since $\{E_\alpha\}$ partitions the $2k$ arcs, $\sum_\alpha|E_\alpha|=2k$, giving $\sum_\alpha(|E_\alpha|-1)=2k-k'=C_t-1$.
\end{proof}

As a consequence, we have the following index-theoretic bound on the total decrement.

\begin{corollary}\label{cor:nod_dom_dec_est}
We have
\[
0\ \le\ \nu(\varphi_0)-\nu(\varphi_t)\ \le\ \sum_{p \in \mathcal{S}_0}|\ind_p(\nabla \varphi_0)| \ -\ \sum_{q \in \mathcal{S}_t} |\ind_q(\nabla \varphi_t)|.
\]
\end{corollary}

\begin{proof}
The global graph $\mathcal G_t$ is obtained from $\mathcal G_0$ by processing the balls $B(p_i,r)$ one at a time: each ball induces a quotient on the current incidence graph, reducing the number of connected components by at most its local decrement. Therefore the total drop satisfies $\nu(\varphi_0)-\nu(\varphi_t)\le\sum_i\Delta_i$.

Fix $p\in \mathcal S_0$ with vanishing order $k$. By the forest identity \eqref{eqn: index_connected_component_relation},
\[
\sum_{q\in \mathcal S_t\cap B(p,r)} |\ind_q(\nabla\varphi_t)| \ =\ k - C_t,
\]
so $C_t-1 = (k-1) - \sum_{q\in \mathcal S_t\cap B(p,r)}|\ind_q(\nabla\varphi_t)| = |\ind_p(\nabla\varphi_0)| - \sum_{q\in \mathcal S_t\cap B(p,r)}|\ind_q(\nabla\varphi_t)|$.
By Lemma~\ref{lem:sector_partition}, the local decrement at $B(p,r)$ is at most $C_t-1$, hence at most $|\ind_p(\nabla\varphi_0)| - \sum_{q\in \mathcal S_t\cap B(p,r)}|\ind_q(\nabla\varphi_t)|$.
Summing over $p\in\mathcal S_0$ yields the claim.
\end{proof}

\section{Strengthening on surfaces}\label{sec:quant_branchfree}

\subsection{Branch-free upper semicontinuity and auxiliary lemmas}\label{sec:analytic_wall}
In this section we record a branch-free formulation of the upper semicontinuity at the level of spectral clusters, so one does not need to select a single simple eigenbranch.

\begin{theorem}[Upper semicontinuity for spectral clusters; branch-free over eigenfunctions]\label{thm:cluster_USC}
Let $(M,g_t)_{t\in(-t_0,t_0)}$ be a $C^\infty$ family on a closed surface with $g_t\to g_0$ in $C^\infty$. Fix $k$ and let $\lambda_{k,0}$ have multiplicity $m\ge1$. Choose $\eta>0$ so that $[\lambda_{k,0}-\eta,\lambda_{k,0}+\eta]$ contains no other spectrum of $(M,g_0)$. For each $t$, let
\[
\mathcal E_t := \{u\in C^\infty(M,\RR):\ -\Delta_{g_t}u=\lambda u,\ \lambda\in{\rm Spec}(-\Delta_{g_t})\cap[\lambda_{k,0}-\eta,\lambda_{k,0}+\eta],\ \|u\|_{L^2}=1\}.
\]
Suppose $u_t\in \mathcal E_t$ and $u_t\to u_0\in \mathcal E_0$ in $C^\infty$. Then
\[
\limsup_{t\to 0}\ \nu(u_t)\ \le\ \nu(u_0).
\]
In particular, $t\mapsto \sup\{\nu(u): u\in \mathcal E_t\}$ is upper semicontinuous at $t=0$.
\end{theorem}

\begin{proof}
After shrinking $t_0$ if necessary, the spectral window $[\lambda_{k,0}-\eta,\lambda_{k,0}+\eta]$ contains exactly $m$ eigenvalues of $-\Delta_{g_t}$ (counting multiplicity). By standard perturbation theory (see \cite{KriegenMichorRainer11}), one may choose a smooth orthonormal frame $\{e_1(t),\dots,e_m(t)\}$ for the associated real spectral subspace. Every $u_t\in\mathcal E_t$ can be written as
\[
  u_t\ =\ \sum_{j=1}^m a_j(t)e_j(t),\qquad \sum_{j=1}^m |a_j(t)|^2=1.
\]
Along any sequence $t_n\to 0$, compactness of the unit sphere in $\RR^m$ yields a subsequence (still denoted $t_n$) with $a_j(t_n)\to a_j^*$ for each $j$, and $\sum|a_j^*|^2=1$. Define $u_0^*:=\sum_{j=1}^m a_j^*e_j(0)$. Since $e_j(t)\to e_j(0)$ in $C^\infty$, we have $u_{t_n}\to u_0^*$ in $C^\infty$. Since we also assumed $u_t\to u_0$, uniqueness of the $C^\infty$-limit gives $u_0^*=u_0$. In particular, $u_0\in\mathcal E_0$ (it lies in the eigenspace of $\lambda_{k,0}$). By Remark~\ref{rem:sequential_USC} (the sequential version of the inequality in Theorem~\ref{thm:main_result_A}), $\nu(u_{t_n})\le\nu(u_0)$ for $n\gg 1$. Since this holds for every sequence $t_n\to 0$, $\limsup_{t\to 0}\nu(u_t)\le\nu(u_0)$.

For the ``in particular'' statement, let $t_n\to 0$ be a sequence with $\sup_{u\in\mathcal E_{t_n}}\nu(u)\to\limsup_{t\to 0}\sup_{u\in\mathcal E_t}\nu(u)$. Since $\nu$ is integer-valued and bounded above by $k+m-1$ (by Courant's theorem), the supremum over $\mathcal E_{t_n}$ is actually a maximum; choose $v_n\in\mathcal E_{t_n}$ achieving it. By the same compactness argument, a subsequence of $v_n$ converges in $C^\infty$ to some $v_0\in\mathcal E_0$. By the sequential inequality, $\nu(v_n)\le\nu(v_0)\le\sup_{u\in\mathcal E_0}\nu(u)$ for $n\gg 1$. Hence $\limsup_{t\to 0}\sup_{u\in\mathcal E_t}\nu(u)\le\sup_{u\in\mathcal E_0}\nu(u)$.
\end{proof}

\begin{Lemma}[Critical points are confined to a shrinking core]\label{lem:NoCritCreation}
Let $u_0$ be a solution of a second-order elliptic equation on a surface near $p$, with a nodal critical point of vanishing order $m\ge2$ at $p$. Then there exist $r_0,c_0>0$ such that
\[
\|\nabla u_0(x)\|_{g_0}\ \ge\ c_0d_{g_0}(x,p)^{m-1}\qquad\text{for all }x\in \overline{B(p,r_0)}\setminus\{p\}.
\]
Consequently, if $\|u_t-u_0\|_{C^1(\overline{B(p,r_0)})}\le \varepsilon$ (norms with respect to a fixed background metric), then every critical point of $u_t$ in $B(p,r_0)$ lies in $B\big(p,(\varepsilon/c_0)^{1/(m-1)}\big)$. In particular, for every fixed $\rho\in(0,r_0)$ there exists $\delta>0$ such that $\mathcal C(u_t)\cap\big(B(p,r_0)\setminus B(p,\rho)\big)=\varnothing$ whenever $\|u_t-u_0\|_{C^1}<\delta$.
\end{Lemma}

\begin{proof}
Since $u_0$ solves a second-order elliptic equation, the Bers--Cheng expansion gives, in isothermal coordinates $(r,\theta)$ centred at $p$,
$u_0(r,\theta)= ar^m\cos m(\theta-\theta_0) + O(r^{m+1})$ with $a\neq0$.
Hence $\|\nabla u_0(x)\|\ge c_0d(x,p)^{m-1}$ for $d(x,p)$ small. If $\nabla u_t(x)=0$ at some $x$ with $d(x,p)=\rho$, then $c_0\rho^{m-1}\le\|\nabla u_0(x)\|=\|\nabla u_0(x)-\nabla u_t(x)\|\le\varepsilon$, so $\rho\le (\varepsilon/c_0)^{1/(m-1)}$.
\end{proof}

We also record two auxiliary lemmas used later.

\begin{Lemma}[Boundary transversality persistence: Dirichlet]\label{lem:Dirichlet_transversality}
Let $u_t$ be a Dirichlet eigenbranch on a surface with smooth boundary. If a boundary point $x_0\in\partial M$ satisfies
\[
\partial_\nu u_0(x_0)=0\quad\text{and}\quad \partial_\tau \partial_\nu u_0(x_0)\neq 0
\]
(i.e.\ $x_0$ is a simple zero of $s\mapsto\partial_\nu u_0(s,0)$), then for small $|t|$ there exists a unique $x_t\in\partial M$ with $x_t\to x_0$ where an interior nodal arc of $u_t$ hits $\partial M$ transversely. When all boundary zeros of $s\mapsto\partial_\nu u_0(s,0)$ are simple, the total number of transverse boundary hits is locally constant in $t$.
\end{Lemma}

\begin{proof}
Choose boundary normal coordinates $(\sigma,\rho)$ near $x_0$, where $\sigma$ is the tangential boundary coordinate and $\rho\ge 0$ is the inward normal coordinate. Since
\[
u_t(\sigma,0)\equiv 0,
\]
Taylor expansion in $\rho$ gives
\[
u_t(\sigma,\rho)=\rho\bigl(a_t(\sigma)+\rho R_t(\sigma,\rho)\bigr),
\]
where
\[
a_t(\sigma):=\partial_\rho u_t(\sigma,0)
\]
and $R_t$ is smooth. (If $\nu$ denotes the outward unit normal, then
$a_t=-\partial_\nu u_t|_{\partial M}$, so zeros and simplicity are unchanged by this sign.)

Set
\[
G(t,\sigma,\rho):=a_t(\sigma)+\rho R_t(\sigma,\rho).
\]
Let $\sigma_0$ be the boundary coordinate of $x_0$. Since $x_0$ is a simple zero of the boundary function $\sigma\mapsto a_0(\sigma)$, we have
\[
G(0,\sigma_0,0)=0,
\qquad
\partial_\sigma G(0,\sigma_0,0)=a_0'(\sigma_0)\neq 0.
\]
Hence the implicit function theorem yields a smooth function
$\sigma=\sigma(t,\rho)$
defined for $(t,\rho)$ near $(0,0)$ such that
$\sigma(0,0)=\sigma_0$ and
$G\bigl(t,\sigma(t,\rho),\rho\bigr)=0$.
Therefore the interior part of the nodal set of $u_t$ near $x_0$ is the smooth arc
\[
\gamma_t(\rho):=\bigl(\sigma(t,\rho),\rho\bigr),
\qquad
\rho\ge 0 \text{ small},
\]
while the boundary itself is the other factor of the zero set coming from the Dirichlet condition. The boundary-hit point is
$x_t=(\sigma_t,0)$ with $\sigma_t:=\sigma(t,0)$,
equivalently the unique zero of $a_t$ near $\sigma_0$.

Since
$\gamma_t'(0)=\bigl(\partial_\rho \sigma(t,0),1\bigr)$
has nonzero inward-normal component, the arc meets $\partial M$ transversely. Since $a_t\to a_0$ in $C^1(\partial M)$, simple zeros persist uniquely and remain finite in number on the compact boundary, so summing gives local constancy of the total number of transverse boundary hits.
\end{proof}

\begin{Lemma}[Boundary transversality persistence: Neumann]\label{lem:Neumann_transversality}
Let $u_t$ be a Neumann eigenbranch on a surface with smooth boundary. If a boundary point $x_0\in\partial M$ satisfies
\[
u_0(x_0)=0\quad\text{and}\quad \partial_\tau u_0(x_0)\neq 0
\]
(i.e.\ $x_0$ is a simple zero of the boundary trace $u_0|_{\partial M}$), then for small $|t|$ there exists a unique $x_t\in\partial M$ with $x_t\to x_0$ such that $u_t(x_t)=0$ and $\partial_\tau u_t(x_t)\neq 0$, and the nodal set of $u_t$ meets $\partial M$ transversely at $x_t$. When all boundary zeros of $u_0|_{\partial M}$ are simple, the total number of boundary-hit points of $\mathcal N(u_t)$ is locally constant in $t$.
\end{Lemma}

\begin{proof}
Choose boundary normal coordinates $(\sigma,\rho)$ near $x_0$, where $\sigma$ is the tangential boundary coordinate and $\rho\ge 0$ is the inward normal coordinate. The Neumann condition gives
\[
\partial_\rho u_t(\sigma,0)=0,
\]
hence
\[
u_t(\sigma,\rho)=b_t(\sigma)+\rho^2 R_t(\sigma,\rho),
\qquad
b_t(\sigma):=u_t(\sigma,0),
\]
with $R_t$ smooth.

Set
$G(t,\sigma,\rho):=b_t(\sigma)+\rho^2 R_t(\sigma,\rho)$.
Let $\sigma_0$ be the boundary coordinate of $x_0$. Since $x_0$ is a simple zero of the boundary trace, we have
\[
G(0,\sigma_0,0)=0,
\qquad
\partial_\sigma G(0,\sigma_0,0)=b_0'(\sigma_0)\neq 0.
\]
Hence the implicit function theorem yields a smooth function
$\sigma=\sigma(t,\rho)$
defined for $(t,\rho)$ near $(0,0)$ such that
$\sigma(0,0)=\sigma_0$ and
$G\bigl(t,\sigma(t,\rho),\rho\bigr)=0$.
Thus the nodal set of $u_t$ near $x_0$ is the smooth arc
\[
\gamma_t(\rho):=\bigl(\sigma(t,\rho),\rho\bigr),
\qquad
\rho\ge 0 \text{ small}.
\]
Differentiating the identity
$G\bigl(t,\sigma(t,\rho),\rho\bigr)=0$
with respect to $\rho$ and evaluating at $\rho=0$ gives
\[
b_t'(\sigma_t)\partial_\rho\sigma(t,0)=0,
\qquad
\sigma_t:=\sigma(t,0),
\]
so $\partial_\rho\sigma(t,0)=0$ because $b_t'(\sigma_t)\neq 0$ for small $|t|$. Hence
$\gamma_t'(0)=(0,1)$,
so the arc meets $\partial M$ transversely. Since $b_t\to b_0$ in $C^1(\partial M)$, simple zeros persist uniquely and remain finite in number on the compact boundary, so summing gives local constancy of the total number of boundary-hit points.
\end{proof}

\begin{theorem}[Openness of Courant-sharpness]\label{thm:CS_open}
Let $g_*$ be a smooth metric on a closed surface $M$ for which the first $k$ eigenvalues are simple, $\nu(\varphi_{\ell,*})=\ell$ for $\ell\le k$, and each $\varphi_{\ell,*}$ has no nodal critical points. Then there exists a $C^2$-neighbourhood $\mathcal U$ of $g_*$ such that, for all $g\in\mathcal U$, the first $k$ eigenvalues remain simple and $\nu(\varphi_\ell(g))=\ell$ for $\ell\le k$.
\end{theorem}

\begin{proof}
Since the first $k$ eigenvalues of $-\Delta_{g_*}$ are simple, standard perturbation theory (see \cite{KriegenMichorRainer11}) gives that for $g$ in a $C^2$-neighbourhood $\mathcal U_1$ of $g_*$, the first $k$ eigenvalues of $-\Delta_g$ remain simple and the corresponding eigenfunctions $\varphi_\ell(g)$ satisfy $\|\varphi_\ell(g)-\varphi_{\ell,*}\|_{C^1(M)}\to 0$ as $g\to g_*$.

Fix $\ell\le k$. Since $\varphi_{\ell,*}$ has no nodal critical points, $0$ is a regular value of $\varphi_{\ell,*}$. By $C^1$-closeness, $0$ remains a regular value of $\varphi_\ell(g)$ for $g\in\mathcal U_1$, the nodal hypersurface $\mathcal N(\varphi_\ell(g))$ is ambiently isotopic to $\mathcal N(\varphi_{\ell,*})$, and hence $\nu(\varphi_\ell(g))=\nu(\varphi_{\ell,*})=\ell$. Taking $\mathcal U:=\mathcal U_1$ completes the proof.
\end{proof}

\subsection{Wavelength rigidity: no sub-wavelength loop creation}\label{sec:mesoscopic}

In this section we show that no connected component of the nodal set can be entirely contained in a ball at the sub-wavelength scale.

\begin{theorem}[No sub-wavelength closed components]\label{thm:meso_no_loops}
Let $(M,g_t,\Pot_t)$ be a $C^\infty$ family of Schr\"{o}dinger data with $g_t\to g_0$ in $C^2$ and $\Pot_t\to\Pot_0$ in $C^0$. Let $\varphi_{k,t}$ be an eigenbranch of $\Ht=-\Delta_{g_t}+\Pot_t$ with $\lambda_{k,0}>0$ and $\varphi_{k,t}\to\varphi_{k,0}$ in $C^\infty$. There exist constants $c_*,t_*>0$ (depending on $(M,g_0,\Pot_0)$ and $\varphi_{k,0}$) such that for all $|t|<t_*$, no connected component of $\mathcal N(\varphi_{k,t})$ is entirely contained in a $g_0$-geodesic ball of radius $c_*\lambda_{k,0}^{-1/2}$.
\end{theorem}

\begin{proof}
Suppose for contradiction that for a sequence $t_j\to 0$ there exist connected components $\Gamma_j\subset\mathcal N(\varphi_{k,t_j})$ entirely contained in $B_{g_0}(x_j, r_0)$ with $r_0:=c_*\lambda_{k,0}^{-1/2}$.

Since $\Gamma_j$ is an \emph{entire} connected component of $\mathcal N(\varphi_{k,t_j})$ and is contained in the open ball $B_{g_0}(x_j,r_0)$, the set $\Gamma_j$ does not meet $\partial B_{g_0}(x_j,r_0)$. For $r_0$ below the injectivity radius of $(M,g_0)$, the ball $B_{g_0}(x_j,r_0)$ is a topological disc. The nodal set $\Gamma_j$ is a compact connected embedded graph (with even-degree vertices) inside this disc. Since every vertex has even degree $\ge 2$ and $\Gamma_j$ is connected, $\Gamma_j$ contains at least one cycle (an Eulerian subgraph has this property; alternatively, a connected graph with all vertices of degree $\ge 2$ contains a cycle). By the Jordan curve theorem applied to this cycle in the simply connected ball, $\Gamma_j$ bounds at least one compact complementary region. Since $\varphi_{k,t_j}$ changes sign across each smooth arc of $\Gamma_j$, at least one sign region is compactly contained in the ball. Among all nodal domains of $\varphi_{k,t_j}$ contained in $B_{g_0}(x_j,r_0)$, choose one; call it $D_j$.

By Lemma~\ref{lem:small-domain-Schro},
\[
\lambda_{k,t_j}-\inf_{D_j}\Pot_{t_j}\ \ge\ \lambda_1^\Delta(D_j;g_{t_j}).
\]
Since $g_t\to g_0$ in $C^2$ on the compact manifold $M$, there exists $c'>0$ (depending only on the curvature of $g_0$ and $r_0$, but not on $j$ or $x_j$) such that $c'^{-1}g_{\mathrm{Euc}}\le g_{t_j}\le c'g_{\mathrm{Euc}}$ in normal coordinates on every ball of radius $r_0$, uniformly for $j$ large. The Euclidean Faber--Krahn inequality gives
\[
\lambda_1^\Delta(D_j;g_{t_j})\ \ge\ (c')^{-1}\lambda_1^\Delta(D_j;g_{\mathrm{Euc}})\ \ge\ \frac{(c')^{-1}\pi j_{0,1}^2}{\mathrm{Area}_{g_{\mathrm{Euc}}}(D_j)}\ \ge\ \frac{C_{\mathrm{FK}}}{\mathrm{Area}_{g_{t_j}}(D_j)},
\]
where $C_{\mathrm{FK}}:=(c')^{-2}\pi j_{0,1}^2>0$ and the last step uses $\mathrm{Area}_{g_{\mathrm{Euc}}}(D_j)\le c'\mathrm{Area}_{g_{t_j}}(D_j)$. Since $\mathrm{Area}_{g_{t_j}}(D_j)\le \mathrm{Area}_{g_{t_j}}(B_{g_0}(x_j,r_0))\lesssim r_0^2 = c_*^2\lambda_{k,0}^{-1}$ and $\inf\Pot_{t_j}$ is uniformly bounded below, choosing $c_*$ small enough forces $\lambda_{k,t_j}\ge(1+\delta)\lambda_{k,0}$ for some $\delta>0$, contradicting $\lambda_{k,t}\to\lambda_{k,0}$.
\end{proof}

\section{Nodal domain count under localised perturbation and applications}\label{sec:changing_topology}

We now deal with the case of localised perturbations in the sense of \cite{Kom2005,Takahashi2002,MukherjeeSaha2021} et al., where the topology of the perturbed surface is not necessarily the same as the limiting surface. More precisely, we consider the following setup.

Let $M_t$ be a one-parameter family of perturbations of the Riemannian surface $M_0$ where the spectrum of $M_0$ is simple. Let $t_i\to 0$ be a sequence such that for all $k\ge 1$ we have 
\begin{equation}\label{eq:eigenval_conv}
    \lim_{t_i \to 0} \lambda_k(M_{t_i}) = \lambda_k(M_0).
\end{equation}

\paragraph{Identification on the unperturbed region.}
We assume there exists a nonempty open set $U\subset M_0$ and, for all $|t|$ small, smooth embeddings
\[
\iota_t: U \hookrightarrow M_t,\qquad \iota_0=\mathrm{id}_U,
\]
such that $\iota_t\to \iota_0$ in $C^\infty$ on compact subsets of $U$. We henceforth identify $U$ with its image $\iota_t(U)\subset M_t$ and write $M:=U$. Heuristically, $M$ is the \emph{eventually unperturbed} part of $M_0$.

Let $\varphi_{k,0}$ be the $k$-th eigenfunction of $-\Delta$ on $M_0$ and $\varphi_{k,t}$ the corresponding $k$-th eigenfunction of the perturbation $M_t$.
We additionally assume $C^\infty$-convergence of eigenfunctions on the unperturbed region:
\begin{equation}\label{eq:eigenfunc_conv}
    \varphi_{k,t}\circ\iota_t \longrightarrow \varphi_{k,0} \quad\text{in } C^\infty_{\mathrm{loc}}(M).
\end{equation}
For ease of notation, we will drop $k$ from our notations and use $(\lambda_t, \varphi_t)$ as our eigenpair corresponding to $M_t$. %, and assume $t_0 = 0$. 
By the convergence assumption \eqref{eq:eigenfunc_conv}, the restriction of $\NNN({\varphi_t})$ to $M$ lies in a $\delta$-neighbourhood of $\NNN(\varphi_0)\cap M$ for sufficiently small $t$. First, we record an analogue of Theorem \ref{thm:main_result_A} above.

\begin{theorem}\label{thm:perturb_nodal_domain_reduction}
Let $M_0$ be a closed surface and $M_t$ a localised perturbation as above. Assume that $M$ is a relatively compact subsurface of $M_0$ with smooth boundary, and that for $|t|$ small:
\begin{enumerate}
\item[(a)] $0$ is a regular value of $\varphi_{0}|_M$ and of $\varphi_{t}|_M$;
\item[(b)] there exists a fixed compact set $T\Subset\mathrm{int}(M)$ such that $\mathcal N(\varphi_0)\subset T$ and $\mathcal N(\varphi_t)\subset T$ for all small $|t|$;
\item[(c)] each connected component of $M_0\setminus M$ meets at most one boundary component of $\partial M$.
\end{enumerate}
Then, for $|t|$ small,
\beq\label{ineq:4.1_gen}
\nu(\varphi_t)\ \le\ \nu(\varphi_0).
\eeq
If, in addition, each connected component of $M_t\setminus M$ also meets at most one boundary component of $\partial M$ (no bridging on either side), then
\beq\label{ineq:4.1_no_nod_crit}
\nu(\varphi_t)\ =\ \nu(\varphi_0).
\eeq
\end{theorem}

\begin{proof}
Write $N_{\mathrm{in}}:=\#\pi_0(M\setminus\mathcal N(\varphi_0))$.

\smallskip\noindent\emph{Step 1: interior component count is preserved.}
By (a)-(b), the nodal sets $\mathcal N(\varphi_0)$ and $\mathcal N(\varphi_t)$ are both contained in the fixed compact set $T\Subset\mathrm{int}(M)$, and $0$ is a regular value of both $\varphi_0|_M$ and $\varphi_t|_M$. Choose a compact neighbourhood $K$ of $T$ with $T\Subset\mathrm{int}(K)\Subset K\Subset\mathrm{int}(M)$. Since $\varphi_t\to\varphi_0$ in $C^\infty_{\mathrm{loc}}(\mathrm{int}(M))$ (the standing convergence assumption), in particular $\varphi_t\to\varphi_0$ in $C^1(K)$. Because $0$ is a regular value of $\varphi_0$ and both nodal sets lie in $T\Subset\mathrm{int}(K)$, $C^1$-closeness on $K$ and the interpolation/isotopy argument of Proposition~\ref{prop:nodal_line_no_crossing} (applied on $K$, with both nodal sets compactly contained in $\mathrm{int}(K)$) give an ambient isotopy of $K$ carrying $\mathcal N(\varphi_0)$ to $\mathcal N(\varphi_t)$, supported in $\mathrm{int}(K)$. In particular, $\#\pi_0(K\setminus\mathcal N(\varphi_t))=\#\pi_0(K\setminus\mathcal N(\varphi_0))$, and since the isotopy is the identity near $\partial K$, the sign pattern of $\varphi_t$ on $\partial K$ agrees with that of $\varphi_0$.

Now consider the collar $C:=M\setminus\mathrm{int}(K)$. Both $\varphi_0$ and $\varphi_t$ are nonzero on $C$ (since $T\subset\mathrm{int}(K)$), so each has constant sign on each connected component of $C$. Every connected component of $C$ touches $\partial K$, and on $\partial K$ the isotopy is the identity, so $\varphi_t$ agrees in sign with $\varphi_0$ there (by $C^1$-convergence on $K$). Since the sign is constant on each connected component of $C$ and agrees with that of $\varphi_0$ on $\partial K$, the region $C$ induces the same connectivity relation among the boundary components of $K\setminus\NNN(\varphi_t)$ as among those of $K\setminus\NNN(\varphi_0)$. It follows that
\[
\#\pi_0(M\setminus\mathcal N(\varphi_t))\ =\ \#\pi_0(M\setminus\mathcal N(\varphi_0))\ =\ N_{\mathrm{in}}.
\]
(The buffer $K\Subset\mathrm{int}(M)$ rather than $K=M$ is used because the standing convergence $\varphi_t\to\varphi_0$ holds in $C^\infty_{\mathrm{loc}}(\mathrm{int}(M))$ but not up to $\partial M$; the collar $C$ provides a zero-free region on which sign comparisons do not require boundary regularity.)

\smallskip\noindent\emph{Step 2: the limiting side.}
Since $\mathcal N(\varphi_0)\Subset\mathrm{int}(M)$, there is a collar of $\partial M$ on which $\varphi_0$ has constant nonzero sign; hence each boundary component of $\partial M$ lies in the closure of exactly one component of $M\setminus\mathcal N(\varphi_0)$. By hypothesis~(c), each connected component of $M_0\setminus M$ meets at most one boundary component, so it attaches to exactly one sign region and causes no merging of interior components. Therefore
\[
\nu(\varphi_0)\ =\ N_{\mathrm{in}}.
\]

\smallskip\noindent\emph{Step 3: the perturbed side.}
Since $\mathcal N(\varphi_t)\subset T\Subset\mathrm{int}(M)$, the exterior $M_t\setminus M$ is zero-free. Every connected component of $M_t\setminus M$ is therefore sign-definite for $\varphi_t$, and gluing it to $M\setminus\mathcal N(\varphi_t)$ can only identify pre-existing sign regions of the same sign; it cannot create new nodal domains. Therefore
\[
\nu(\varphi_t)\ \le\ \#\pi_0(M\setminus\mathcal N(\varphi_t))\ =\ N_{\mathrm{in}}\ =\ \nu(\varphi_0).
\]

\smallskip\noindent\emph{Step 4: equality under no-bridging.}
If each component of $M_t\setminus M$ also meets at most one boundary component, then (since $\mathcal N(\varphi_t)\Subset\mathrm{int}(M)$, each boundary component carries a single sign of $\varphi_t$) the $t$-side exterior also attaches to exactly one sign region and causes no merging. Hence $\nu(\varphi_t)=N_{\mathrm{in}}=\nu(\varphi_0)$.
\end{proof}

\begin{remark}
Theorem~\ref{thm:perturb_nodal_domain_reduction} is deliberately stated in a regular-value form, in order to isolate the zero-free exterior gluing mechanism from the singular local analysis developed earlier in the paper. A more general \emph{localised upper semicontinuity} statement can be obtained when $\varphi_0$ has nodal critical points in $\mathrm{int}(M)$: remove small balls around these nodal critical points, apply the regular-value isotopy argument on the complement of those balls in $M$, use the local sector-counting result of Theorem~\ref{thm:loc_nod_est}(c) in each singular ball, run the dual-incidence-graph quotient argument inside the fixed core $M$ as in the proof of Theorem~\ref{thm:main_result_A}, and then glue the zero-free exterior $M_t\setminus M$ as in the proof above. This yields $\nu(\varphi_t)\le\nu(\varphi_0)$ without assuming that $0$ is a regular value on all of $M$, and in particular covers collapsing connected-sum degenerations $M_1\#_\varepsilon M_2$ (with $M_2$ collapsing) once the fixed core contains the nodal critical points of the limit.

The equality conclusion~\eqref{ineq:4.1_no_nod_crit} under no-bridging is, however, specific to the regular-value setting. Once nodal critical points are allowed in the core, the nodal-domain count may already drop inside the singular balls, so one should only expect the upper semicontinuity statement in that generality. For the Courant-sharp application in Theorem~\ref{thm:Freitas_app} below, the regular-value version stated above suffices: the Freitas model eigenfunctions on the inner disc have nodal sets consisting of disjoint embedded circles, so hypothesis~(a) holds automatically on the chosen core.
\end{remark}

\subsection{Application: nodal data prescriptions}\label{subsec:nod_da_pres}
We first construct closed surfaces which are Courant-sharp up to any finite level.

\begin{proof}[Proof of  Theorem \ref{thm:Freitas_app}]
    Consider a closed surface $M$ of genus $m$ with any Riemannian metric $g_0$. Let $B(p,r)$ be a geodesic disc inside $M$ centred at $p$ with radius $r$, where $r$ is smaller than the injectivity radius of $(M,g_0)$. Fix a diffeomorphism $\Psi:\mathbb D\to B(p,r)$. Now recall the following theorem from \cite{Freitas2002}.
    
    \begin{theorem}[Freitas]\label{thm:Freitas}
    Given any positive integer $k$ there exists a family $\mathcal{M}$ of rotationally invariant metrics on the unit disc with positive curvature and fixed area, for which the first $k$ eigenvalues of the Laplace operator with both Neumann and Dirichlet boundary conditions are simple. In both cases, the eigenfunction corresponding to the $l$-th eigenvalue $l=(2, \cdots, k)$ has $l-1$ nodal lines which are closed disjoint circles dividing the disc into $l$ nodal domains.
    \end{theorem}

    Let $g_D$ be one of the metrics guaranteed by Freitas' result, and transplant it to the geodesic disc via $\Psi_*(r^2g_D)$. Let $\varphi_{l,0}$ be the $l$-th Neumann eigenfunction of $-\Delta_{\Psi_*(r^2 g_D)}$ on $B(p,r)$, with eigenvalue $\lambda_{l,0}$. By the Freitas theorem, the first $k$ eigenvalues are simple. Since the Freitas eigenfunctions are $S^1$-invariant with nodal set consisting of $l-1$ disjoint circumferences, $0$ is a regular value on the disc model; after pushforward by $\Psi$, the nodal set of $\varphi_{l,0}$ consists of $l-1$ pairwise disjoint embedded circles in $B(p,r)$. Choose a fixed $\delta_*>0$ smaller than the minimum distance from any of these circles to $\partial B(p,r)$, so that $\mathcal N(\varphi_{l,0})\Subset B(p,r-\delta_*)$ for all $l\le k$.
    
    Define the piecewise metric
    \[
    \tilde g_\varepsilon\ =\ \begin{cases}
    \Psi_*(r^2g_D) & \text{on } B(p,r),\\
    \varepsilon^2g_0 & \text{on } M\setminus B(p,r).
    \end{cases}
    \]
    By \cite[Section~2, Steps~2-3]{EPS2015}, the first $k$ eigenvalues of the quadratic form associated with $\tilde g_\varepsilon$ converge to $\lambda_{l,0}$ as $\varepsilon\to 0$, and the corresponding eigenfunctions converge in $C^\infty$ on compact subsets of $B(p,r)$ to $\varphi_{l,0}$. By \cite[Section~2, Step~4]{EPS2015}, for each fixed small $\varepsilon$ there exist smooth metrics $g_\varepsilon$ approximating $\tilde g_\varepsilon$ and agreeing with $\Psi_*(r^2g_D)$ on $B(p,r-\delta_*)$, such that the first $k$ eigenpairs of $-\Delta_{g_\varepsilon}$ remain arbitrarily close to those of $\tilde g_\varepsilon$. Since the limiting Neumann eigenvalues $\lambda_{1,0}<\cdots<\lambda_{k,0}$ are pairwise distinct (by Freitas' theorem), the minimum gap $\min_{i\neq j}|\lambda_{i,0}-\lambda_{j,0}|>0$; choosing $\varepsilon$ small and then the smoothing close enough, we may arrange simultaneously for all $l\le k$ that $\lambda_{l,\varepsilon}$ is simple and
    \[
    \varphi_{l,\varepsilon}|_{B(p,r-\delta_*)}\to \varphi_{l,0}|_{B(p,r-\delta_*)} \quad\text{in }C^\infty.
    \]
    
    We now count nodal domains. For each $l\le k$, the limiting eigenfunction $\varphi_{l,0}$ has no nodal critical points, and $0$ is a regular value of both $\varphi_{l,0}$ and $\varphi_{l,\varepsilon}$ on $B(p,r-\delta_*)$ (the latter by $C^\infty$-convergence, for $\varepsilon$ small). Since $\mathcal N(\varphi_{l,0})\Subset B(p,r-\delta_*)$, the regular-value stability argument of Proposition~\ref{prop:nodal_line_no_crossing} (applied on the fixed inner disc $B(p,r-\delta_*)$) gives that $\mathcal N(\varphi_{l,\varepsilon})\cap B(p,r-\delta_*)$ is ambiently isotopic to $\mathcal N(\varphi_{l,0})\cap B(p,r-\delta_*)$ for $\varepsilon$ small. In particular, the $l-1$ nested circles of $\varphi_{l,0}$ persist as $l-1$ disjoint embedded loops of $\varphi_{l,\varepsilon}$ inside $B(p,r-\delta_*)$. Since these loops lie inside the topological disc $B(p,r-\delta_*)$ and are isotopic to the original nested circles there, $\varphi_{l,\varepsilon}$ changes sign across each loop, and the disc already contains $l$ alternating sign regions. Each such region is contained in a distinct nodal domain of $\varphi_{l,\varepsilon}$ on $M$, giving $\nu(\varphi_{l,\varepsilon})\ge l$. By Courant's theorem, $\nu(\varphi_{l,\varepsilon})\le l$, so
    \[
    \nu(\varphi_{l,\varepsilon})\ =\ l.
    \]
    Setting $g:=g_\varepsilon$ completes the proof.
\end{proof}

\begin{remark}
The proof above invokes Proposition~\ref{prop:nodal_line_no_crossing} directly on the inner disc $B(p,r-\delta_*)$, rather than routing through Theorem~\ref{thm:perturb_nodal_domain_reduction}. Another approach would be to generalise Theorem~\ref{thm:perturb_nodal_domain_reduction} to a compact $M_0$ with smooth (possibly nonempty) boundary equipped with Neumann conditions and then verify the three hypotheses (a)-(c). We do not carry this out here.
\end{remark}

As our next application, we address the problem of prescribing nodal intersections on the boundary of surfaces. Before dealing with the more general Theorem \ref{thm: prescription_nodal_intersection_multiple_bdry}, we look at a simpler case of surfaces with one boundary component. The following result allows us to prescribe the number of nodal intersections at the boundary of a compact ``flask''.

\begin{theorem}\label{thm:prescription_nodal_intersection_boundary}
    Let $M$ be a compact Riemannian surface whose boundary $\pa M$ has exactly one component. Given $n \in \mathbb{N}$, there exists a metric $g$ on $M$ such that for some $k\in \NN$ the nodal set of the corresponding Neumann eigenfunction $\varphi_k$ intersects the boundary $\pa M$ exactly $2n$ times.  
\end{theorem}

\begin{proof}
We first construct a topological disc carrying a \emph{simple} Neumann eigenfunction with exactly $2n$ transverse boundary hits, and then glue that disc to $M$.

\smallskip\noindent\emph{Step 1: a simple disc eigenfunction with exactly $2n$ boundary hits.}
We construct a smooth metric $h$ on $\mathbb D$ and a simple Neumann eigenfunction $u$ of $(\mathbb D,h)$ whose nodal set meets $\partial\mathbb D$ transversely in exactly $2n$ points.

Consider conformal metrics $g_f:=e^{2f(r)}g_{\mathrm{Euc}}$ on $\mathbb D$, where $f:[0,1]\to\mathbb R$ is a smooth radial profile with $f(1)=0$ and all odd-order derivatives vanishing at $0$ (so that $r\mapsto f(r)$ extends to a smooth even function and $g_f$ is a smooth Riemannian metric on $\mathbb D$). Every such metric preserves the $S^1$-symmetry, so the Neumann problem decomposes into angular modes: for each $m\ge 0$, the $m$-th mode yields a Sturm--Liouville problem on $[0,1]$ with Neumann boundary conditions, whose eigenvalues $\lambda_{m,1}(f)<\lambda_{m,2}(f)<\cdots$ are simple (within the $m$-th mode) and depend analytically on $f$. For $m\ge 1$, each $\lambda_{m,k}(f)$ has multiplicity $2$ in the full Neumann spectrum (from $\cos(m\theta)$ and $\sin(m\theta)$), unless it accidentally coincides with an eigenvalue from a different angular mode.

The first eigenfunction of the $n$-th radial mode has the form $v_f(r)\cos(n\theta)$, where $v_f$ is the first eigenfunction of the corresponding Sturm--Liouville problem. Since $v_f$ satisfies a second-order ODE with Neumann condition $v_f'(1)=0$, and $v_f\not\equiv 0$, ODE uniqueness gives $v_f(1)\neq 0$. Hence the boundary trace $v_f(1)\cos(n\theta)$ has exactly $2n$ simple zeros on $\partial\mathbb D$.

For the Euclidean metric ($f=0$), the eigenvalue $\lambda_{n,1}(0)=(j'_{n,1})^2$ may accidentally coincide with eigenvalues from other angular modes. We eliminate such coincidences by a transversality argument: for each fixed pair $(m,k)$ with $m\neq n$, the condition $\lambda_{n,1}(f)=\lambda_{m,k}(f)$ defines a closed subset of the space of smooth radial profiles. This subset has empty interior, because the first-order variation of $\lambda_{m,k}$ under $f\mapsto f+\varepsilon\phi$ is a weighted $L^2$-pairing of $\phi$ with $|R_{m,k}|^2$ (the squared radial eigenfunction for the $(m,k)$-mode), and $|R_{n,1}|^2$ and $|R_{m,k}|^2$ are linearly independent when $m\neq n$: near $r=0$ the regular $m$-th angular mode satisfies $R_{m,k}(r)=c_{m,k}r^m+O(r^{m+2})$ with $c_{m,k}\neq 0$, so $|R_{m,k}|^2\sim |c_{m,k}|^2r^{2m}$ and $|R_{n,1}|^2\sim |c_{n,1}|^2r^{2n}$; since $m\neq n$ these have different vanishing orders at $0$ and cannot be proportional. Hence one can always choose $\phi$ so that $\lambda_{n,1}$ and $\lambda_{m,k}$ move at different rates. Since there are only countably many pairs $(m,k)$ with $m\neq n$, the Baire category theorem gives a residual (hence dense) set of radial profiles $f$ for which $\lambda_{n,1}(f)$ does not coincide with any eigenvalue from another angular mode. Choose such an $f_0$ close to $0$; then $\lambda_{n,1}(f_0)$ has multiplicity exactly $2$ in the Neumann spectrum of $(\mathbb D,g_{f_0})$, with eigenspace $E_{f_0}=\mathrm{span}\{v_{f_0}(r)\cos(n\theta),v_{f_0}(r)\sin(n\theta)\}$.

By generic simplicity of the Neumann spectrum (\cite{Uhlenbeck1976}), there exists a sequence of smooth metrics $h_j\to g_{f_0}$ in $C^\infty(\overline{\mathbb D})$ such that each $(\mathbb D,h_j)$ has simple Neumann spectrum. For $j$ large, the Riesz spectral projection of $-\Delta_{h_j}$ onto the eigenvalues near $\lambda_{n,1}(f_0)$ has rank $2$ and its range converges to $E_{f_0}$. By simplicity of the spectrum of $h_j$, these are exactly two simple eigenvalues $\lambda_j^-<\lambda_j^+\to\lambda_{n,1}(f_0)$. Choose $u_j$ to be a normalised eigenfunction for either $\lambda_j^-$ or $\lambda_j^+$.

We claim that, for all sufficiently large $j$, the boundary trace $u_j|_{\partial\mathbb D}$ has exactly $2n$ simple zeros. By elliptic compactness, after passing to a subsequence, $u_j\to v$ in $C^\infty(\overline{\mathbb D})$, where $v$ is a nonzero element of $E_{f_0}$ (since $u_j$ lies in the range of the rank-$2$ Riesz projection converging to $E_{f_0}$). Every nonzero $v\in E_{f_0}$ has boundary trace of the form $v|_{\partial\mathbb D}=c\cos(n(\theta-\theta_0))$ for some $c\neq 0$ (since $v_{f_0}(1)\neq 0$), hence has exactly $2n$ simple zeros. By $C^1$-convergence on $\partial\mathbb D$, the same holds for $u_j|_{\partial\mathbb D}$ for all large $j$, proving the claim.

Fix such a large $j_0$, let $D:=(\mathbb D,h_{j_0})$, and let $u:=u_{j_0}$. Then $u$ is a simple Neumann eigenfunction on $D$, and the boundary trace $u|_{\partial D}$ has exactly $2n$ simple zeros, so $\mathcal N(u)$ meets $\partial D$ transversely at exactly $2n$ points.

\smallskip\noindent\emph{Step 2: gluing $D$ to $M$.}
Let $(M,g_M)$ be a compact Riemannian surface with one boundary component $\partial M$. Choose a point $x_1\in D\setminus\mathcal N(u)$ and $\epsilon>0$ so small that $B(x_1,\epsilon)\cap\mathcal N(u)=\emptyset$. Form the glued surface
\[
M_\epsilon\ :=\ \bigl(D\setminus B(x_1,\epsilon)\bigr)\cup_{\phi_\epsilon} M,
\]
where $\phi_\epsilon:\partial B(x_1,\epsilon)\to\partial M$ is a smooth attaching map. By the collar neighbourhood theorem, $M_\epsilon$ is diffeomorphic to $M$; fix a diffeomorphism $F_\epsilon:M\to M_\epsilon$ carrying $\partial M$ to $\partial D$. Equip $M_\epsilon$ with a smooth metric $\tilde g_\epsilon$ equal to $h_{j_0}$ on $D\setminus B(x_1,2\epsilon)$ and $\epsilon^2 g_M$ on $M$, smoothly interpolated in a collar. The unique boundary component is the untouched outer boundary $\partial D$. Pull back: $g_\epsilon:=F_\epsilon^*\tilde g_\epsilon$.

\smallskip\noindent\emph{Step 3: spectral degeneration and persistence of the boundary zeros.}
By \cite[Section~2, Steps~2-4]{EPS2015}, as $\epsilon\to 0$ the Neumann spectrum of $(M_\epsilon,\tilde g_\epsilon)$ converges to the Neumann spectrum of $(D,h_{j_0})$, and for the simple eigenvalue of $u$ there exists an index $k\in\mathbb N$ such that, after choosing signs, the $k$-th Neumann eigenfunction $\varphi_{k,\epsilon}$ satisfies $\varphi_{k,\epsilon}\to u$ in $C^\infty$ on compact subsets of $D\setminus\{x_1\}$. In particular, the convergence is $C^\infty$ in a fixed collar of the boundary $\partial D$.

Let $b:=u|_{\partial D}$. Since $b$ has exactly $2n$ simple zeros on the compact curve $\partial D$, there exist pairwise disjoint arcs $I_1,\dots,I_{2n}\subset\partial D$ around those zeros and a constant $c>0$ such that
\[
|b'|\ge c \quad\text{on each } I_j,
\qquad
|b|\ge c \quad\text{on } \partial D\setminus \bigcup_{j=1}^{2n} I_j.
\]
Because $\varphi_{k,\epsilon}\to u$ in $C^\infty$ near $\partial D$, the boundary traces $\varphi_{k,\epsilon}|_{\partial D}$ converge to $b$ in $C^1(\partial D)$. Therefore, for all sufficiently small $\epsilon$, each arc $I_j$ contains exactly one zero of $\varphi_{k,\epsilon}|_{\partial D}$ and there are no other zeros on $\partial D$. At each such zero, $\partial_\tau(\varphi_{k,\epsilon}|_{\partial D})\neq 0$ (by $C^1$-closeness to $b$), and the Neumann condition $\partial_\nu\varphi_{k,\epsilon}=0$ forces the interior nodal set to cross $\partial D$ transversely. Hence $|\mathcal N(\varphi_{k,\epsilon})\cap\partial M|=2n$. Setting $g:=g_\epsilon$ completes the proof.
\end{proof}

As our final result, we now generalise the above theorem to surfaces with multiple boundary components. 

\begin{proof}[Proof of Theorem \ref{thm: prescription_nodal_intersection_multiple_bdry}]
\emph{Step 1: choosing the disc models.}
For each $i\in\{1,\dots,b\}$, apply Step~1 of the proof of Theorem~\ref{thm:prescription_nodal_intersection_boundary}. This gives a topological disc $(D_i,h_i)$ with a \emph{simple} Neumann eigenpair $-\Delta_{h_i}u_i=\alpha_i^{(0)}u_i$ such that the boundary trace $u_i|_{\partial D_i}$ has exactly $2n_i$ simple zeros. Let
\[
0=\lambda^{(i)}_0<\lambda^{(i)}_1<\lambda^{(i)}_2<\cdots
\]
be the simple Neumann spectrum of $(D_i,h_i)$, and let $m_i\ge 1$ be the index for which $\lambda^{(i)}_{m_i}=\alpha_i^{(0)}$. Define the spectral-gap ratios
\[
\rho_i^-:=\frac{\lambda^{(i)}_{m_i-1}}{\lambda^{(i)}_{m_i}}<1,
\qquad
\rho_i^+:=\frac{\lambda^{(i)}_{m_i+1}}{\lambda^{(i)}_{m_i}}>1.
\]

\smallskip\noindent\emph{Step 2: arranging a contiguous isolated block of limit eigenvalues.}
Choose $r>1$ such that $\rho_i^-r<1<r<\rho_i^+$ for every $i=1,\dots,b$. This is possible because each $\rho_i^-<1<\rho_i^+$. Next choose $A>0$ and pairwise distinct numbers $A<\beta_1<\beta_2<\cdots<\beta_b<rA$. Rescale the metric $h_i$ by a constant factor $c_i^2$ so that the distinguished eigenvalue becomes $\beta_i$, i.e.\ $\alpha_i^{(0)}/c_i^2=\beta_i$. (Under such a rescaling all Neumann eigenvalues are multiplied by $c_i^{-2}$, and the boundary-hit count of $u_i$ is unchanged.) After this rescaling (still denoted $h_i$), the distinguished eigenvalue on $D_i$ is $\beta_i$, while every lower Neumann eigenvalue of $D_i$ satisfies
\[
\frac{\lambda^{(i)}_{m_i-1}}{c_i^2}=\rho_i^-\beta_i\le\rho_i^-rA<A,
\]
and every higher Neumann eigenvalue satisfies $\lambda^{(i)}_{m_i+1}/c_i^2=\rho_i^+\beta_i>\rho_i^+A>rA$.

Therefore the ordered union of the Neumann spectra of the discs has no eigenvalue in $[A,rA]$ except $\beta_1,\dots,\beta_b$. Denoting this ordered union by $(\mu_k)$, there exists $l\in\mathbb N$ such that $\mu_{l+i}=\beta_i$ for $i=1,\dots,b$.

\smallskip\noindent\emph{Step 3: gluing the discs to the given surface.}
Let $(M,g_M)$ be the given Riemannian surface with boundary components $B_1,\dots,B_b$. For each $i$, choose $x_i\in D_i\setminus\mathcal N(u_i)$ with $B(x_i,\epsilon)\cap\mathcal N(u_i)=\emptyset$. Attach $D_i\setminus B(x_i,\epsilon)$ to $B_i$ by a smooth gluing map, producing a surface $M_\epsilon$. By the collar neighbourhood theorem, $M_\epsilon$ is diffeomorphic to $M$; fix a diffeomorphism $F_\epsilon:M\to M_\epsilon$ carrying each boundary component $B_i$ of $M$ to the outer boundary $\partial D_i$. Define a smooth metric on $M_\epsilon$ by
\[
\tilde g_\epsilon\ =\ \begin{cases}
h_i & \text{on } D_i\setminus B(x_i,2\epsilon),\quad i=1,\dots,b,\\
\epsilon^2g_M & \text{on the core } M,
\end{cases}
\]
with smooth interpolation in the gluing collars. Pull back to $M$: $g_\epsilon:=F_\epsilon^*\tilde g_\epsilon$.

\smallskip\noindent\emph{Step 4: spectral degeneration and localisation.}
We use the following consequence of \cite[Section~2, Steps~2-4]{EPS2015}: as $\epsilon\to 0$, the Neumann eigenvalues of $(M_\epsilon,\tilde g_\epsilon)$ converge to the ordered union $(\mu_k)$ of the Neumann spectra of the disconnected limit $\bigsqcup_i(D_i,h_i)$; moreover, if a limit eigenvalue $\mu_k$ is simple and isolated, then after choosing signs the corresponding normalised eigenfunction converges in $C^\infty$ on compact subsets of the relevant disc $D_i$ away from the gluing point to the model eigenfunction on $D_i$. Since the block $\mu_{l+1}=\beta_1<\cdots<\mu_{l+b}=\beta_b$ is isolated from the rest of the limit spectrum, for all sufficiently small $\epsilon$ the global eigenvalues $\lambda_{l+1,\epsilon}<\cdots<\lambda_{l+b,\epsilon}$ are simple with $\lambda_{l+i,\epsilon}\to\beta_i$. After choosing signs, the $(l+i)$-th normalised Neumann eigenfunction $\psi_{i,\epsilon}$ on $(M_\epsilon,\tilde g_\epsilon)$ satisfies $\psi_{i,\epsilon}\to u_i$ in $C^\infty$ on compact subsets of $D_i\setminus\{x_i\}$, and in particular in a fixed collar of $B_i=\partial D_i$.

\smallskip\noindent\emph{Step 5: persistence of the boundary zeros.}
Fix $i$. The boundary trace $u_i|_{B_i}$ has exactly $2n_i$ simple zeros. Hence there exist pairwise disjoint arcs $I_{i,1},\dots,I_{i,2n_i}\subset B_i$ around these zeros and a constant $c_i>0$ such that
\[
|\partial_\tau(u_i|_{B_i})|\ge c_i\quad\text{on each }I_{i,j},
\qquad
|u_i|\ge c_i\quad\text{on }B_i\setminus\bigcup_{j=1}^{2n_i}I_{i,j}.
\]
Because $\psi_{i,\epsilon}\to u_i$ in $C^\infty$ near $B_i$, the boundary traces $\psi_{i,\epsilon}|_{B_i}$ converge to $u_i|_{B_i}$ in $C^1(B_i)$. Therefore, for all sufficiently small $\epsilon$, each arc $I_{i,j}$ contains exactly one zero of $\psi_{i,\epsilon}|_{B_i}$ and there are no other zeros on $B_i$. At each such zero, $\partial_\tau(\psi_{i,\epsilon}|_{B_i})\neq 0$, and the Neumann condition forces the interior nodal set to cross $B_i$ transversely. Hence $|\mathcal N(\psi_{i,\epsilon})\cap B_i|=2n_i$ for $i=1,\dots,b$.
\end{proof}

\begin{remark}[Payne-type stability]
The openness of the Payne property (nodal arc of the second Dirichlet eigenfunction hitting the boundary transversely, with no interior closed loop) under $C^2$-perturbation of strictly convex domains was established in \cite{MukherjeeSaha2022}; see also \cite[Section~3]{MukherjeeSaha2021} for further discussion and extensions to non-convex domains.
\end{remark}

\section{Higher-dimensional reduction and conjectural extension}\label{sec:higher_dim}

Throughout this section, $\varphi_t$ denotes the eigenfunction family in the standing Schr\"odinger framework of the paper (an eigenfunction of $-\Delta_{g_t}+\Pot_t$ on $(M^n,g_t)$), with $\varphi_t\to\varphi_0$ in $C^\infty(M)$. We will use the standard consequences of this framework: $\NNN(\varphi_t)$ is $(n-1)$-rectifiable, $\mathcal S(\varphi_t)$ has codimension at least $2$, $\mathcal C(\varphi_t)$ has empty interior (\cite{HardtSimon1989}), and the Bers--Cheng asymptotic expansion at nodal critical points produces tame finite chamber decompositions (see the notation paragraph below) on small spheres around points of $\mathcal S(\varphi_t)$.

The proof of Theorem~\ref{thm:main_result_A} uses three genuinely planar inputs:
the boundary intersection count on circles, the forest/Euler identity in a disc,
and the resulting sector-partition argument. In dimensions $n\ge 3$, none of these
have literal analogues. The singular set
\[
\mathcal S(\varphi):=\{x\in M:\varphi(x)=0,\ \nabla\varphi(x)=0\}
\]
may have dimension as large as $n-2$, while the link
\[
\NNN(\varphi)\cap \partial B(p,r)
\]
is the nodal set of the boundary trace $\varphi|_{\partial B(p,r)}$ on $S^{n-1}$, in general disconnected and with singular subset of codimension at least $2$, rather than the finite planar graph $\NNN(\varphi)\cap B(p,r)$ used in dimension $2$.

What nevertheless survives in all dimensions is the \emph{global} quotient mechanism.
Once one knows that, near each singular point, the chamber decomposition on a small
boundary sphere can only \emph{coarsen} under perturbation, the same incidence-graph
argument gives upper semicontinuity of the nodal domain count. We make this precise below.

\medskip

We first record the regular-value case, which is completely dimension-free.

\begin{proposition}[Regular-value stability in all dimensions]\label{prop:regular_value_all_dim}
Let $(M^n,g_t)$ be a $C^\infty$ family of closed Riemannian manifolds, and let
$\varphi_t\to \varphi_0$ in $C^\infty(M)$. If $0$ is a regular value of $\varphi_0$,
then for all sufficiently small $|t|$, $0$ is a regular value of $\varphi_t$,
the nodal hypersurface $\NNN(\varphi_t)$ is ambiently isotopic to $\NNN(\varphi_0)$,
and
\[
\nu(\varphi_t)=\nu(\varphi_0).
\]
\end{proposition}

\noindent This is the dimension-free analogue of Proposition~\ref{prop:nodal_line_no_crossing}; in dimension $2$ the two coincide.

\begin{proof}
Choose an open neighbourhood $U$ of $\NNN(\varphi_0)$ and a constant $c>0$ such that
\[
|\nabla \varphi_0|\ge 2c \quad \text{on } U,
\qquad
|\varphi_0|\ge 2c \quad \text{on } M\setminus U.
\]
This is possible because $0$ is a regular value of $\varphi_0$ and $M$ is compact.
By $C^1$-convergence on $U$ and $C^0$-convergence on $M\setminus U$, for all sufficiently
small $|t|$ we have
\[
|\nabla \varphi_t|\ge c \quad \text{on } U,
\qquad
|\varphi_t|\ge c \quad \text{on } M\setminus U.
\]
Hence $\NNN(\varphi_t)\subset U$ and $0$ is a regular value of $\varphi_t$.

For each fixed small $|t|$, define the interpolation $u_s:=(1-s)\varphi_0+s\varphi_t$ for $s\in[0,1]$. Since $\|\varphi_t-\varphi_0\|_{C^1(U)}<c$, we have $|\nabla u_s|\ge c$ on $U$ and $|u_s|\ge c$ on $M\setminus U$ for all $s\in[0,1]$, so $0$ is a regular value of every $u_s$ and $\NNN(u_s)\subset U$. The map $(x,s)\mapsto u_s(x)$ is smooth on $M\times[0,1]$, so $\Sigma:=\{(x,s)\in M\times[0,1]:u_s(x)=0\}$ is a smooth hypersurface and $\pi:\Sigma\to[0,1]$ is a proper submersion. By Ehresmann's fibration theorem, $\pi$ is a locally trivial fibre bundle; by the isotopy extension theorem (see e.g.\ \cite[Theorem~1.3]{Hirsch1976}), the resulting fibrewise diffeomorphism extends to an ambient isotopy of $M$ carrying $\NNN(\varphi_0)$ to $\NNN(\varphi_t)$.
Therefore $M\setminus \NNN(\varphi_t)$ is homeomorphic to $M\setminus\NNN(\varphi_0)$, and
$\nu(\varphi_t)=\nu(\varphi_0)$ for small $|t|$.
\end{proof}

\medskip

The singular case is subtler. The right higher-dimensional local datum is not an isotopy
statement for the full nodal set inside a punctured ball, but a statement about the
\emph{chamber decomposition} on a small boundary sphere.

\paragraph{Notation: chamber decomposition.}
Let $\{p_1,\dots,p_N\}\subset M$ be a finite set, and let $\overline{B_i}\subset M$ be pairwise disjoint closed geodesic balls centred at $p_i$. For each $i$ and $t$, set
\[
\mathcal B_i(t):=\pi_0\bigl(\partial B_i\setminus \NNN(\varphi_t)\bigr),\qquad
\mathcal Q_i(t):=\pi_0\bigl(B_i\setminus \NNN(\varphi_t)\bigr).
\]
Elements of $\mathcal B_i(t)$ are called \emph{boundary regions}; elements of $\mathcal Q_i(t)$ \emph{local chambers}. For $Q\in\mathcal Q_i(t)$, the set of \emph{boundary regions incident to $Q$} is
\[
E(Q):=\{A\in\mathcal B_i(t): A\subset \overline Q\cap \partial B_i\}.
\]
For each boundary region $A\in\mathcal B_i(t)$ there exists a unique local chamber
$Q\in\mathcal Q_i(t)$ such that $A\subset \overline Q\cap \partial B_i$,
equivalently, $A\in E(Q)$. Indeed, for every $a\in A$, continuity and the fact that
$\varphi_t(a)\neq 0$ provide an open neighbourhood $U_a\subset A$ and
$\varepsilon_a>0$ such that the inward collar
\[
C_a:=\{\exp_x(s\eta_i(x)):\ x\in U_a,\ 0<s<\varepsilon_a\}\subset B_i,
\]
where $\eta_i$ is the inward unit normal to $\partial B_i$, is contained in
$\{\varphi_t\neq 0\}$. Hence $C_a$ lies in a unique local chamber
$Q_a\in\mathcal Q_i(t)$. If $U_a\cap U_b\neq\emptyset$, then
$C_a\cap C_b\neq\emptyset$, so $Q_a=Q_b$. Since $A$ is connected, all $Q_a$
coincide; denote the common chamber by $Q$. Consequently
$\{E(Q):Q\in\mathcal Q_i(t)\}$
is a partition of $\mathcal B_i(t)$.

\begin{proposition}[Higher-dimensional USC reduction via chamber coarsening]\label{prop:abstract_USC_highdim}
Assume $\mathcal S(\varphi_0)=\{p_1,\dots,p_N\}$ is finite, and choose pairwise disjoint closed geodesic balls $\overline{B_i}$ centred at $p_i$, small enough that
\begin{enumerate}
\item[(a)] $\partial B_i$ is transverse to $\NNN(\varphi_0)$, and
\item[(b)] $\mathrm{vol}_{g_0}(B_i)$ is below the Faber--Krahn threshold of Lemma~\ref{lem:small-domain-Schro} at eigenvalue $\lambda_0$.
\end{enumerate}
Set $M':=M\setminus \bigcup_{i=1}^N B_i$, and use the chamber-decomposition notation $\mathcal B_i(t),\mathcal Q_i(t),E(Q)$ above. Assume further that for all sufficiently small $|t|$:
\begin{itemize}
\item[\rm(iii)] the partition $\{E(Q):Q\in\mathcal Q_i(t)\}$ is a coarsening of the partition $\{E(Q):Q\in\mathcal Q_i(0)\}$ under a natural identification $\mathcal B_i(t)\cong\mathcal B_i(0)$ induced by the standing framework (see the proof).
\end{itemize}
Then
\[
\nu(\varphi_t)\le \nu(\varphi_0)
\]
for all sufficiently small $|t|$.
\end{proposition}

\begin{proof}
We begin by deriving two intermediate facts from the standing Schr\"odinger framework; the coarsening hypothesis~(iii) is then the only remaining input.

\smallskip\noindent
\emph{Step 1: boundary-relative regular-value isotopy on $M'$.} Since $\mathcal S(\varphi_0)\subset\bigcup_i B_i$, we have $\NNN(\varphi_0)\cap M'\subset \{\nabla\varphi_0\neq 0\}$, so $0$ is a regular value of $\varphi_0|_{M'}$. By (a), $\partial M'=\bigsqcup_i\partial B_i$ is transverse to $\NNN(\varphi_0)$, and $C^1$-closeness of $\varphi_t$ to $\varphi_0$ preserves both the regular-value property and the transversality for small $|t|$. The boundary-relative version of the regular-value isotopy argument of Proposition~\ref{prop:regular_value_all_dim} (using relative Ehresmann and the relative isotopy extension theorem) then produces a diffeomorphism $H_t:M'\to M'$ sending each $\partial B_i$ to itself setwise, carrying $\NNN(\varphi_0)\cap M'$ onto $\NNN(\varphi_t)\cap M'$, and carrying each connected component $\Omega_0$ of $M'\setminus\NNN(\varphi_0)$ onto a connected component $\Omega_t:=H_t(\Omega_0)$ of $M'\setminus\NNN(\varphi_t)$ on which $\varphi_t$ has the same sign as $\varphi_0$ has on $\Omega_0$. In particular, $H_t$ induces a bijection $\mathcal B_i(0)\leftrightarrow\mathcal B_i(t)$; this is the identification used to state~(iii).

\smallskip\noindent
\emph{Step 2: every local chamber reaches the boundary.} We claim $E(Q)\neq\emptyset$ for every $Q\in\mathcal Q_i(t)$. Suppose not. Then $\overline Q\cap \partial B_i\subset\NNN(\varphi_t)\cap\partial B_i$, so $\partial Q\subset\NNN(\varphi_t)$ and $Q$ is a nodal domain of $\varphi_t$ on all of $M$ with $Q\subset B_i$. But $\mathrm{vol}_{g_t}(Q)\le\mathrm{vol}_{g_t}(B_i)\to \mathrm{vol}_{g_0}(B_i)$ as $t\to 0$, and (b) together with the uniform bounds $\lambda_t\to\lambda_0$, $\|\Pot_t\|_{L^\infty}$ bounded, places this below the Faber--Krahn threshold of Lemma~\ref{lem:small-domain-Schro} for all small $|t|$, a contradiction.

\smallskip\noindent
\emph{Step 3: graph-quotient argument.} For each fixed $t$, define a bipartite graph $\mathcal G_t$:
\begin{itemize}
\item \emph{Exterior vertices:} the connected components of $M'\setminus \NNN(\varphi_t)$.
\item \emph{Local vertices:} the chambers in $\bigsqcup_{i=1}^N \mathcal Q_i(t)$.
\item \emph{Edges:} an exterior vertex $E$ and a local vertex $Q\in \mathcal Q_i(t)$ are joined whenever they share a boundary region $A\in E(Q)$ along which $\varphi_t$ has the same sign on both sides.
\end{itemize}

By Step 2, each local vertex is incident to at least one edge. No nodal domain $D$ of $\varphi_t$ on $M$ is contained entirely in one ball $B_i$: if $D\subset B_i$, then $D=Q$ for some $Q\in\mathcal Q_i(t)$, and by Step 2 the chamber $Q$ reaches some boundary region $A$ where $\varphi_t\neq 0$; a path in $Q$ concatenated with a path through $A$ into $M'$ stays in $\{\varphi_t\neq 0\}$, showing $D$ extends into $M'$, a contradiction.

Hence every nodal domain of $\varphi_t$ intersects $M'$ and at least one ball $B_i$ in pieces forming a connected subgraph of $\mathcal G_t$; conversely, each connected component of $\mathcal G_t$ lies in a single nodal domain (adjacent pieces share boundary points where $\varphi_t$ has a fixed sign). Therefore
\[
\nu(\varphi_t)=\#\pi_0(\mathcal G_t).
\]

By Step 1, the exterior vertices of $\mathcal G_t$ are naturally identified with those of $\mathcal G_0$, and for each $i$ the set of boundary regions $\mathcal B_i(t)$ is identified with $\mathcal B_i(0)$. Under these identifications, (iii) says that each block $E(Q_t)$ for $Q_t\in\mathcal Q_i(t)$ is a union of blocks $E(Q_0)$ of the partition of $\mathcal B_i(0)$; equivalently, the local vertices of $\mathcal G_t$ inside $B_i$ are obtained from those of $\mathcal G_0$ by merging $Q_0,Q_0'\in\mathcal Q_i(0)$ whenever $E(Q_0)\cup E(Q_0')\subset E(Q_t)$ for some $Q_t$. The exterior vertices are unchanged. Thus $\mathcal G_t$ is a vertex-quotient of $\mathcal G_0$, and such quotients cannot increase the number of connected components. Therefore
\[
\nu(\varphi_t)=\#\pi_0(\mathcal G_t)\le \#\pi_0(\mathcal G_0)=\nu(\varphi_0).\qedhere
\]
\end{proof}

\begin{remark}\label{rem:no_local_isotopy_highdim}
One should \emph{not} expect ambient isotopy of the full local nodal set near a singular point
in dimensions $n\ge 3$. For instance, in $\mathbb R^3$ the nodal set
$\{xy=0\}$ (two coordinate planes meeting along the $z$-axis)
is smoothed by the perturbation
$\{xy+tz=0\}$.
On a small sphere, the link goes from two great circles (four chambers) to one smooth curve (two chambers): a genuine coarsening, not an isotopy. Proposition~\ref{prop:abstract_USC_highdim} isolates exactly
this weaker, and for our purposes sufficient, local input.
\end{remark}

\medskip

This leads to the following conjectural extension of Theorem~\ref{thm:main_result_A}
to isolated singularities in higher dimensions.

\begin{conjecture}[Isolated conical singularities imply USC]\label{conj:isolated_conical_USC}
Let $n\ge 3$, let $(M^n,g_t)$ be a $C^\infty$ family of closed Riemannian manifolds with
$g_t\to g_0$, and let $\varphi_{k,t}\to \varphi_{k,0}$ in $C^\infty(M)$. Assume that
$\mathcal S(\varphi_{k,0})$ is finite.

For each $p\in \mathcal S(\varphi_{k,0})$, assume that there exist:
\begin{itemize}
\item a homogeneous harmonic polynomial $H_p$ of degree $m_p\ge 2$;
\item a scale $r_p>0$;
\item a number $\theta_p\in(0,1)$;
\item and a pinching constant $\delta_p>0$ sufficiently small;
\end{itemize}
such that:

\begin{enumerate}
\item[(a)] the normalised blow-ups of $\varphi_{k,0}$ at $p$ converge to $H_p$; more precisely,
if
\[
u_{p,r}(x)
:=
\frac{\varphi_{k,0}(\exp_p(rx))}
{\left(\int_{\partial B_1}
\varphi_{k,0}(\exp_p(ry))^2d\sigma(y)\right)^{1/2}},
\]
then
\[
u_{p,r}\longrightarrow H_p
\]
in $C^1_{\mathrm{loc}}(B_1\setminus\{0\})$ as $r\downarrow 0$;

\item[(b)] $0$ is a regular value of $H_p|_{S^{n-1}}$;

\item[(c)] the generalised Almgren frequency at $p$ is uniformly pinched across scales:
\[
N_p(r)-N_p(\theta_p r)\le \delta_p
\qquad \text{for all } 0<r<r_p.
\]
\end{enumerate}

Then
\[
\limsup_{t\to 0}\nu(\varphi_{k,t})\le \nu(\varphi_{k,0}).
\]
\end{conjecture}

\begin{remark}\label{rem:what_local_theorem_should_say}
The point of Conjecture~\ref{conj:isolated_conical_USC} is that hypotheses (a)-(c) should
force the following local statement at each singular point $p$: for some fixed small radius
$\rho_p>0$, the hypersurface $\NNN(\varphi_{k,t})\cap \partial B(p,\rho_p)$ is a small perturbation
of $\NNN(\varphi_{k,0})\cap \partial B(p,\rho_p)$, every local chamber of
$B(p,\rho_p)\setminus \NNN(\varphi_{k,t})$ is incident to at least one boundary region on $\partial B(p,\rho_p)$, and the induced
partition of
\[
\pi_0\bigl(\partial B(p,\rho_p)\setminus \NNN(\varphi_{k,t})\bigr)
\]
is a coarsening of the corresponding partition at $t=0$. Once this is known,
Proposition~\ref{prop:abstract_USC_highdim} gives the global upper semicontinuity.
\end{remark}

\begin{remark}[Beyond isolated singularities]\label{rem:fully_stratified_case}
In general, the singular set in dimensions $n\ge 3$ need not be discrete. A full
higher-dimensional theory should therefore be formulated on tubular neighbourhoods of
Whitney strata, with boundary sphere bundles replacing the circles used in dimension $2$.
The reduction of Proposition~\ref{prop:abstract_USC_highdim} suggests that the correct local input is a stability/coarsening
statement for chamber decompositions on those sphere bundles, not ambient isotopy of the full
nodal set near the singular stratum.
\end{remark}

\paragraph{Relevant higher-dimensional literature.}
The structural foundation for nodal sets in dimensions $n\ge 3$ is the Hardt--Simon decomposition \cite{HardtSimon1989}: for solutions of second-order elliptic equations with smooth coefficients, the nodal set is a smooth hypersurface away from a singular set of Hausdorff dimension at most $n-2$, and the singular part is countably $(n-2)$-rectifiable. The quantitative theory of the Almgren frequency function and its generalisation to variable-coefficient elliptic equations is developed by Naber--Valtorta \cite{NaberValtorta2017}, who obtain effective pinching estimates, unique approximation by homogeneous harmonic polynomials across scales, and volume bounds on critical sets. Logunov \cite{Logunov2018} establishes propagation-of-smallness and polynomial upper estimates on $\mathcal H^{n-1}(\NNN(\varphi))$ for Laplace eigenfunctions in all dimensions. Together, these tools provide the analytic backdrop for the conjectural local coarsening statement in Remark~\ref{rem:what_local_theorem_should_say}.

\appendix

\section{Technical tools: inner-radius stability}\label{sec:inner_radius}

The following result complements the Faber--Krahn arguments used in the main body by establishing that the inner radius of any nodal domain of $\varphi_t$ is comparable to $1/\sqrt{\lambda_t}$, uniformly for small $|t|$. The argument follows the strategy of Mangoubi \cite{Mangoubi2008} (see also Charron--Mangoubi \cite{CharronMangoubi2024} for the higher-dimensional analogue), adapted to the one-parameter family setting.

\begin{theorem}\label{thm: Inner_radius_perturbation}
Let $M_t:=(M,g_t)$, $t\in\RR$, be a family of smooth closed Riemannian surfaces with
$g_t\to g_0$ in $C^2$ as $t\to 0$. Let $\varphi_t$ be a smooth eigenbranch with
$-\Delta_{g_t}\varphi_t=\lambda_t\varphi_t$, $\lambda_t\to\lambda_0>0$,
and $\varphi_t\to\varphi_0$ in $C^\infty$.
Then there exist constants $a,c>0$ such that for every $|t|<a$ and every nodal domain
$\Omega_t$ of $\varphi_t$,
\begin{equation}\label{ineq: inner radius estimate}
  \inrad_{g_t}(\Omega_t)\ \ge\ \frac{c}{\sqrt{\lambda_t}}.
\end{equation}
In particular, the minimum inner radius
$R_t:=\min_{\Omega_t}\inrad_{g_t}(\Omega_t)$ satisfies $R_t\ge m$ for all $|t|<a$,
where $m>0$ depends only on $\lambda_0$ and the geometric constants of $(M,g_0)$.
\end{theorem}

We first record the key uniformity inputs.

\paragraph{Metric comparability.}
There exist $a_0>0$ and $\kappa\ge 1$ such that for all $x,y\in M$ and all $|t|<a_0$,
\begin{equation}\label{ineq: metric comparability}
  \frac{1}{\kappa}d_0(x,y)\ \le\ d_t(x,y)\ \le\ \kappa d_0(x,y).
\end{equation}
In particular, $D_0(p,r)\subset D_t(p,\kappa r)$ and $D_t(p,r)\subset D_0(p,\kappa r)$
for all $p\in M$ and $|t|<a_0$, where $D_t(p,r)$ denotes the geodesic disc of radius $r$
centred at $p$ in $(M,g_t)$.

From our hypothesis, for $\delta_1>0$ there exists $a_1>0$ (depending on $\delta_1$) such that
\begin{equation}\label{ineq:eigfn_conv}
  \lambda_0-\delta_1<\lambda_t<\lambda_0+\delta_1\qquad\forall t\in(-a_1,a_1).
\end{equation}

We now state the analogues of \cite[Theorem~3.2 and Lemma~3.3]{Mangoubi2008} for the one-parameter family $M_t$ (see also \cite{NazarovPolterovichSodin, Nadirashvili1991}).

\begin{Lemma}[Uniform conformal charts and quasiharmonic factorisation]\label{lem: Mangoubi's lemmas}
After shrinking $a_0>0$, there exist positive constants
\[
  q_+, q_-, \rho, \epsilon, \delta, K_1, K_2
\]
independent of $p\in M$ and of $|t|<a_0$, such that the following hold.
\begin{enumerate}
  \item For each point $p\in M$ and each $|t|<a_0$, there exists a conformal map
    $\Psi_{t,p}:\DD\to D_t(p,\rho)$ with $\Psi_{t,p}(0)=p$ and a positive function
    $q_t$ on $\DD$ satisfying
    \[
      \Psi_{t,p}^*(g_t)=q_t(z)|dz|^2,\qquad q_-<q_t<q_+.
    \]
  \item For every $|t|<a_0$ and every disc $D_t(p,r)$ with
    $r\le \epsilon/\sqrt{\lambda_0+\delta_1}$, the eigenfunction $\varphi_t|_{D_t(p,r)}$
    admits a quasiharmonic factorisation
    \[
      \varphi_t|_{D_t(p,r)}\ =\ (v_t\cdot u_t)\circ h_t,
    \]
    where $h_t:D_t(p,r)\to\DD$ is a $K_t$-quasiconformal homeomorphism with
    $K_1\le K_t\le K_2$, the function $u_t$ is harmonic on $\DD$, and
    $1-\delta\le v_t\le 1$ on $\DD$.
\end{enumerate}
\end{Lemma}

\begin{proof}
Part~(1) follows by introducing $t$-dependence into the argument of
\cite[Lemma~2.3.3]{JurgenJost2002}. The conformal chart $\Psi_{t,p}$ exists for each
$(t,p)$; the bounds $q_\pm$ and the radius $\rho$ may be chosen locally in $p$.
Since $M$ is compact and the family $\{g_t:|t|<a_0\}$ is precompact in $C^2$,
a standard finite-covering argument yields uniform constants
$q_\pm,\rho>0$ independent of $p$ and $t$.

Part~(2) follows from \cite[Lemma~3.3]{NazarovPolterovichSodin} applied to the
one-parameter family of elliptic equations $-\Delta_{g_t}\varphi_t=\lambda_t\varphi_t$,
combined with Part~(1). The constants $\epsilon,\delta,K_1,K_2$ arise from the
Nazarov--Polterovich--Sodin quasiharmonic approximation and depend on $q_\pm$,
$\rho$, and the $C^2$ bound on $g_t$; by the same compactness argument, they can
be chosen uniformly in $p$ and $t$.
\end{proof}

\begin{proof}[Proof of Theorem~\ref{thm: Inner_radius_perturbation}]
Fix any $|t|<a:=\min\{a_0,a_1\}$ (so that both the conformal-chart/factorisation inputs and the eigenvalue bound $\lambda_t<\lambda_0+\delta_1$ are available) and any nodal domain $\Omega_t$ of $\varphi_t$. Without loss of
generality, assume $\varphi_t|_{\Omega_t}>0$. Let $p_t\in\Omega_t$ be a point where
$\varphi_t$ attains its maximum on $\Omega_t$.

\medskip
\noindent\textbf{Step~1 (Quasiharmonic factorisation at $p_t$).}
Apply Lemma~\ref{lem: Mangoubi's lemmas}(2) on a wavelength-scale disc $D_t:=D_t(p_t,r_0)$
with $r_0:=\epsilon/\sqrt{\lambda_0+\delta_1}$. We obtain
\[
  \varphi_t|_{D_t}\ =\ (v_t\cdot u_t)\circ h_t
\]
with $h_t:D_t\to\DD$ a $K_t$-quasiconformal homeomorphism ($K_1\le K_t\le K_2$),
$u_t$ harmonic on $\DD$, and $1-\delta\le v_t\le 1$. After composing $h_t$ with a
disc automorphism, we may arrange $h_t(p_t)=0$.

\medskip
\noindent\textbf{Step~2 (Hotspot comparison).}
Since $h_t(p_t)=0$, we have
\[
  \varphi_t(p_t)\ =\ v_t(0)u_t(0)\ \ge\ (1-\delta)u_t(0).
\]
Let $U^0_t\subset\DD$ be the connected component of $\{u_t>0\}$ containing $0$.
Since $v_t>0$, the preimage $h_t^{-1}(U^0_t)$ lies inside $\Omega_t\cap D_t$, so for every $z\in U^0_t$ the point $h_t^{-1}(z)$ belongs to $\Omega_t$. Because $p_t$ maximises $\varphi_t$ on $\Omega_t$,
\[
  \varphi_t(p_t)\ \ge\ \varphi_t(h_t^{-1}(z))\ =\ v_t(z)u_t(z)\ \ge\ (1-\delta)u_t(z)
\]
for every $z\in U^0_t$. Taking the supremum over $z$ gives $\varphi_t(p_t)\ge (1-\delta)\sup_{U^0_t}u_t$. Since also $v_t\le 1$, we have $\varphi_t(p_t)=v_t(0)u_t(0)\le u_t(0)$, and combining:
\begin{equation}\label{ineq: hotspot harmonic ratio}
  \frac{u_t(0)}{\sup_{U^0_t}u_t}\ \ge\ 1-\delta.
\end{equation}

\medskip
\noindent\textbf{Step~3 (Inscribed disc in the positivity component).}
The estimate \eqref{ineq: hotspot harmonic ratio} says that $u_t(0)$ is within a factor $(1-\delta)$ of $\sup_{U^0_t}u_t$. We claim there exists $r_*=r_*(\delta)>0$ such that $\mathbb D_{r_*}\subset U^0_t$. Indeed, $u_t$ is a positive harmonic function on the connected open set $U^0_t\subset\mathbb D$. The boundary $\partial U^0_t$ decomposes into $\Gamma_1:=\partial U^0_t\cap\mathbb D$ (where $u_t=0$) and $\Gamma_2:=\partial U^0_t\cap\partial\mathbb D$ (where $u_t\le\sup_{U^0_t}u_t$). By the maximum principle on $U^0_t$,
\[
u_t(0)\ \le\ \omega(0,\Gamma_2,U^0_t)\sup_{U^0_t}u_t,
\]
where $\omega(0,\Gamma_2,U^0_t)$ is the harmonic measure of $\Gamma_2$ relative to $0$ in $U^0_t$. Combined with \eqref{ineq: hotspot harmonic ratio}, this gives $\omega(0,\Gamma_1,U^0_t)\le\delta$. By the Beurling--Nevanlinna estimate (see \cite[Section~3]{Mangoubi2008} and \cite{Nadirashvili1991}), the harmonic measure of $\Gamma_1$ from $0$ in $U^0_t$ is bounded below in terms of $\operatorname{dist}(0,\Gamma_1)$, and the bound $\omega(0,\Gamma_1,U^0_t)\le\delta$ forces $\operatorname{dist}(0,\Gamma_1)\ge r_*(\delta)>0$, giving $\mathbb D_{r_*}\subset U^0_t$. Crucially, $\delta$ (hence $r_*$) is uniform in $p_t\in M$, in the choice of $\Omega_t$, and in $|t|<a$.

\medskip
\noindent\textbf{Step~4 (Mori distortion and conclusion).}
Recall that the quasiharmonic factorisation in Step~1 operates on the wavelength-scale disc $D_t=D_t(p_t,r_0)$ with $r_0:=\epsilon/\sqrt{\lambda_0+\delta_1}$.
The conformal chart $\Psi_{t,p_t}$ maps $\DD$ onto $D_t(p_t,\rho)$ with $\rho$ uniform (Lemma~\ref{lem: Mangoubi's lemmas}(1)), so the preimage of $D_t(p_t,r_0)$ under $\Psi_{t,p_t}$ contains a Euclidean disc $\DD_R$ with
\[
  R\ \ge\ \frac{r_0}{\rho}\cdot C_1\ =\ \frac{C_1\epsilon}{\rho\sqrt{\lambda_0+\delta_1}},
\]
for a uniform constant $C_1>0$ depending on the conformal distortion bounds $q_\pm$. In particular, $R\asymp r_0\asymp \lambda_0^{-1/2}$.

Define $\tilde h_t:=h_t\circ\Psi_{t,p_t}:\DD_R\to\DD$. This composition is $K_t$-quasiconformal with $\tilde h_t(0)=0$.
By Mori's distortion theorem (\cite[Ch.~III.C]{Ahlfors-1966}),
\[
  |\tilde h_t(z)|\ \le\ C(K_2)\Big(\frac{|z|}{R}\Big)^{1/K_2}
  \qquad(z\in\DD_R).
\]
Hence the preimage of $\DD_{r_*}$ under $\tilde h_t$ contains $\DD_{c_0 R}$ with
$c_0=c_0(K_2,r_*)>0$. Since $\Psi_{t,p_t}$ is conformal with $\Psi_{t,p_t}^*(g_t)=q_t(z)|dz|^2$ and $q_t\ge q_->0$ (Lemma~\ref{lem: Mangoubi's lemmas}(1)), every point on the image of $\partial\DD_{c_0 R}$ is at $g_t$-distance at least $\sqrt{q_-}\cdot c_0 R$ from $p_t$. Since $\Psi_{t,p_t}(\DD_{c_0R})\subset\Omega_t$, we conclude
\[
  \mathrm{dist}_{g_t}(p_t,\partial\Omega_t)
  \ \ge\ \sqrt{q_-}c_0R
  \ \ge\ \frac{\sqrt{q_-}c_0C_1\epsilon}{\rho\sqrt{\lambda_0+\delta_1}}
  \ =:\ \frac{C}{\sqrt{\lambda_0+\delta_1}}.
\]
To convert this into a bound of the form $c/\sqrt{\lambda_t}$, we use the \emph{lower} eigenvalue bound $\lambda_t\ge\lambda_0-\delta_1$ from \eqref{ineq:eigfn_conv}. Setting $c:=C\sqrt{(\lambda_0-\delta_1)/(\lambda_0+\delta_1)}$ (which is positive for $\delta_1<\lambda_0$), we have
\[
\frac{C}{\sqrt{\lambda_0+\delta_1}}\ =\ \frac{c}{\sqrt{\lambda_0-\delta_1}}\ \ge\ \frac{c}{\sqrt{\lambda_t}}.
\]
Since the centre of a largest inscribed ball in $\Omega_t$ is at least as far from
$\partial\Omega_t$ as the hotspot $p_t$, we conclude
\[
  \inrad_{g_t}(\Omega_t)
  \ \ge\ \mathrm{dist}_{g_t}(p_t,\partial\Omega_t)
  \ \ge\ \frac{c}{\sqrt{\lambda_t}}.
\]
As this holds for every nodal domain $\Omega_t$ and every $|t|<a:=\min\{a_0,a_1\}$
(where $a_1$ ensures $\lambda_0-\delta_1<\lambda_t<\lambda_0+\delta_1$), we obtain the claimed uniform bound. Choosing $\delta_1:=\lambda_0/2$, we get $m:=c/\sqrt{3\lambda_0/2}>0$, so
\[
  R_t\ :=\ \min_{\Omega_t}\inrad_{g_t}(\Omega_t)
  \ \ge\ \frac{c}{\sqrt{\lambda_t}}
  \ \ge\ \frac{c}{\sqrt{3\lambda_0/2}}
  \ =:\ m\ >\ 0
  \qquad\forall|t|<a.\qedhere
\]
\end{proof}

\begin{remark}
For one-parameter families of surfaces with boundary as considered in
Theorem~\ref{thm:prescription_nodal_intersection_boundary}, a similar inner-radius bound is expected, provided one has the corresponding boundary analogue of the Mangoubi inscribed-disc theorem (i.e.\ a uniform inscribed-ball estimate for nodal domains touching the boundary). We do not pursue this here.
\end{remark}

\subsection{Acknowledgements}  The first named author would like to thank the Prime Minister's Research Fellowship (PMRF), India program(Fellowship no. 1301599) for partially supporting his work.  
The first and second named authors would like to thank the Indian Institute of Technology Bombay, India, for providing ideal working conditions. The work is partially funded by the European Union. The third named author would like to thank Institut de Recherche Math\'{e}matique Avanc\'{e}e (IRMA), France, and the ERC grant (Acronym = InSpeGMos, Number = 101096550) for supporting the work. The third named author would also like to thank Nalini Anantharaman and Bernard Helffer for various discussions and suggestions that helped improve the article.

\bibliographystyle{amsalpha}
\bibliography{references}

@book{Ahlfors-1966,
    AUTHOR = {Ahlfors, Lars V.},
     TITLE = {Lectures on quasiconformal mappings},
    SERIES = {Van Nostrand Mathematical Studies, No.\ 10},
 PUBLISHER = {D.\ Van Nostrand Co., Inc., Toronto--New York--London},
      YEAR = {1966},
  MRNUMBER = {200442},
}

@article{Albert1973,
    AUTHOR = {Albert, Jeffrey H.},
     TITLE = {Generic properties of eigenfunctions of elliptic partial differential operators},
   JOURNAL = {Trans.\ Amer.\ Math.\ Soc.},
    VOLUME = {238},
      YEAR = {1978},
     PAGES = {341--354},
  MRNUMBER = {471000},
}

@article{Baer1997,
    AUTHOR = {B\"{a}r, Christian},
     TITLE = {On nodal sets for {D}irac and {L}aplace operators},
   JOURNAL = {Comm.\ Math.\ Phys.},
    VOLUME = {188},
      YEAR = {1997},
    NUMBER = {3},
     PAGES = {709--721},
  MRNUMBER = {1473317},
}

@article{BeckGuptaMarzuola2024-rectangle,
    AUTHOR = {Beck, Thomas and Gupta, Marichi and Marzuola, Jeremy},
     TITLE = {Nodal set openings on perturbed rectangular domains},
   JOURNAL = {Ann.\ Henri Poincar\'{e}},
    VOLUME = {25},
      YEAR = {2024},
    NUMBER = {11},
     PAGES = {4889--4929},
  MRNUMBER = {4809363},
}

@incollection{BH,
    AUTHOR = {B\'{e}rard, Pierre and Helffer, Bernard},
     TITLE = {Nodal sets of eigenfunctions, {A}ntonie {S}tern's results revisited},
 BOOKTITLE = {Actes de S\'{e}minaire de Th\'{e}orie Spectrale et
              G\'{e}om\'{e}trie. Ann\'{e}e 2014--2015},
 PUBLISHER = {Universit\'{e} de Grenoble I, Institut Fourier,
              St.\ Martin d'H\`{e}res},
      YEAR = {2015},
     PAGES = {1--37},
}

@article{BerardHelffer2016-Courant_sharp_1,
    AUTHOR = {B\'{e}rard, Pierre and Helffer, Bernard},
     TITLE = {Courant-sharp eigenvalues for the equilateral torus, and for the equilateral triangle},
   JOURNAL = {Lett.\ Math.\ Phys.},
    VOLUME = {106},
      YEAR = {2016},
    NUMBER = {12},
     PAGES = {1729--1789},
  MRNUMBER = {3569644},
}

@article{BerardHelfferKiwan-Courant_sharp_2,
    AUTHOR = {B\'{e}rard, Pierre and Helffer, Bernard and Kiwan, Rola},
     TITLE = {Courant-sharp property for {D}irichlet eigenfunctions on the {M}\"{o}bius strip},
   JOURNAL = {Port.\ Math.},
    VOLUME = {78},
      YEAR = {2021},
    NUMBER = {1},
     PAGES = {1--41},
  MRNUMBER = {4269391},
}

@article{BerardHelfferKiwan2022-Courant_sharp_3,
    AUTHOR = {B\'{e}rard, Pierre and Helffer, Bernard and Kiwan, Rola},
     TITLE = {Courant-sharp eigenvalues of compact flat surfaces: {K}lein bottles and cylinders},
   JOURNAL = {Proc.\ Amer.\ Math.\ Soc.},
    VOLUME = {150},
      YEAR = {2022},
    NUMBER = {1},
     PAGES = {439--453},
  MRNUMBER = {4335889},
}

@article{berard1982inegalites,
    AUTHOR = {B\'{e}rard, Pierre and Meyer, Daniel},
     TITLE = {In\'{e}galit\'{e}s isop\'{e}rim\'{e}triques et applications},
   JOURNAL = {Ann.\ Sci.\ \'{E}cole Norm.\ Sup.\ (4)},
    VOLUME = {15},
      YEAR = {1982},
    NUMBER = {3},
     PAGES = {513--541},
  MRNUMBER = {690651},
}

@incollection{Bonnaillie-NoelHelffer2016,
    AUTHOR = {Bonnaillie-No\"{e}l, Virginie and Helffer, Bernard},
     TITLE = {Nodal and spectral minimal partitions---the state of the art in 2016},
 BOOKTITLE = {Shape optimization and spectral theory},
 PUBLISHER = {De Gruyter Open, Warsaw},
      YEAR = {2017},
     PAGES = {353--397},
}

@book{Brasselet87,
    AUTHOR = {Brasselet, Jean-Paul and Seade, Jos\'{e} and Suwa, Tatsuo},
     TITLE = {Vector fields on singular varieties},
    SERIES = {Lecture Notes in Mathematics},
    VOLUME = {1987},
 PUBLISHER = {Springer-Verlag, Berlin},
      YEAR = {2009},
}

@article{CharronMangoubi2024,
    AUTHOR = {Charron, Philippe and Mangoubi, Dan},
     TITLE = {The inner radius of nodal domains in high dimensions},
   JOURNAL = {Adv.\ Math.},
    VOLUME = {452},
      YEAR = {2024},
     PAGES = {Paper No.\ 109787, 17},
  MRNUMBER = {4766721},
}

@article{Cheng1976,
    AUTHOR = {Cheng, Shiu Yuen},
     TITLE = {Eigenfunctions and nodal sets},
   JOURNAL = {Comment.\ Math.\ Helv.},
    VOLUME = {51},
      YEAR = {1976},
    NUMBER = {1},
     PAGES = {43--55},
  MRNUMBER = {397805},
}

@article{Courant1923,
    AUTHOR = {Courant, R.},
     TITLE = {Ein allgemeiner {S}atz zur {T}heorie der {E}igenfunktionen selbstadjungierter {D}ifferentialausdr\"{u}cke},
   JOURNAL = {Nachr.\ Ges.\ Wiss.\ G\"{o}ttingen, Math.-Phys.\ Kl.},
    VOLUME = {1923},
      YEAR = {1923},
     PAGES = {81--84},
}

@article{EPS2015,
    AUTHOR = {Enciso, Alberto and Peralta-Salas, Daniel},
     TITLE = {Eigenfunctions with prescribed nodal sets},
   JOURNAL = {J.\ Differential Geom.},
    VOLUME = {101},
      YEAR = {2015},
    NUMBER = {2},
     PAGES = {197--211},
  MRNUMBER = {3399096},
}

@article{Freitas2002,
    AUTHOR = {Freitas, Pedro},
     TITLE = {Closed nodal lines and interior hot spots of the second eigenfunction of the {L}aplacian on surfaces},
   JOURNAL = {Indiana Univ.\ Math.\ J.},
    VOLUME = {51},
      YEAR = {2002},
    NUMBER = {2},
     PAGES = {305--316},
  MRNUMBER = {1909291},
}

@article{ghosh2013nodal,
    AUTHOR = {Ghosh, Amit and Reznikov, Andre and Sarnak, Peter},
     TITLE = {Nodal domains of {M}aass forms {I}},
   JOURNAL = {Geom.\ Funct.\ Anal.},
    VOLUME = {23},
      YEAR = {2013},
    NUMBER = {5},
     PAGES = {1515--1568},
  MRNUMBER = {3102912},
}

@article{HardtSimon1989,
    AUTHOR = {Hardt, Robert and Simon, Leon},
     TITLE = {Nodal sets for solutions of elliptic equations},
   JOURNAL = {J.\ Differential Geom.},
    VOLUME = {30},
      YEAR = {1989},
    NUMBER = {2},
     PAGES = {505--522},
  MRNUMBER = {1010169},
}

@article{Hoffmann-OstenhofMichorNadirashvili1999,
    AUTHOR = {Hoffmann-Ostenhof, T. and Michor, P. W. and Nadirashvili, N.},
     TITLE = {Bounds on the multiplicity of eigenvalues for fixed membranes},
   JOURNAL = {Geom.\ Funct.\ Anal.},
    VOLUME = {9},
      YEAR = {1999},
    NUMBER = {6},
     PAGES = {1169--1188},
  MRNUMBER = {1736932},
}

@book{JurgenJost2002,
    AUTHOR = {Jost, J\"{u}rgen},
     TITLE = {Compact {R}iemann surfaces},
   EDITION = {Second},
 PUBLISHER = {Springer-Verlag, Berlin},
      YEAR = {2002},
      NOTE = {An introduction to contemporary mathematics},
  MRNUMBER = {1909701},
}

@article{jung2016number,
    AUTHOR = {Jung, Junehyuk and Zelditch, Steve},
     TITLE = {Number of nodal domains and singular points of eigenfunctions of negatively curved surfaces with an isometric involution},
   JOURNAL = {J.\ Differential Geom.},
    VOLUME = {102},
      YEAR = {2016},
    NUMBER = {1},
     PAGES = {37--66},
  MRNUMBER = {3447086},
}

@article{jung2018quantum,
    AUTHOR = {Jang, Seung uk and Jung, Junehyuk},
     TITLE = {Quantum unique ergodicity and the number of nodal domains of eigenfunctions},
   JOURNAL = {J.\ Amer.\ Math.\ Soc.},
    VOLUME = {31},
      YEAR = {2018},
    NUMBER = {2},
     PAGES = {303--318},
  MRNUMBER = {3758146},
}

@article{KriegenMichorRainer11,
    AUTHOR = {Kriegl, Andreas and Michor, Peter W. and Rainer, Armin},
     TITLE = {Denjoy--{C}arleman differentiable perturbation of polynomials and unbounded operators},
   JOURNAL = {Integral Equations Operator Theory},
    VOLUME = {71},
      YEAR = {2011},
    NUMBER = {3},
     PAGES = {407--416},
  MRNUMBER = {2852194},
}

@article{Laurent2012,
    AUTHOR = {Bakri, Laurent},
     TITLE = {Critical sets of eigenfunctions of the {L}aplacian},
   JOURNAL = {J.\ Geom.\ Phys.},
    VOLUME = {62},
      YEAR = {2012},
    NUMBER = {10},
     PAGES = {2024--2037},
  MRNUMBER = {2944790},
}

@article{Lewy,
    AUTHOR = {Lewy, Hans},
     TITLE = {On the minimum number of domains in which the nodal lines of spherical harmonics divide the sphere},
   JOURNAL = {Comm.\ Partial Differential Equations},
    VOLUME = {2},
      YEAR = {1977},
    NUMBER = {12},
     PAGES = {1233--1244},
  MRNUMBER = {477199},
}

@article{Lyons2025-rectangle,
    AUTHOR = {Lyons, Andrew},
     TITLE = {Nodal sets of {L}aplacian eigenfunctions with an eigenvalue of multiplicity 2},
   JOURNAL = {Ann.\ Math.\ Qu\'{e}.},
    VOLUME = {49},
      YEAR = {2025},
    NUMBER = {1},
     PAGES = {105--153},
  MRNUMBER = {4894860},
}

@article{Mangoubi2008,
    AUTHOR = {Mangoubi, Dan},
     TITLE = {On the inner radius of a nodal domain},
   JOURNAL = {Canad.\ Math.\ Bull.},
    VOLUME = {51},
      YEAR = {2008},
    NUMBER = {2},
     PAGES = {249--260},
  MRNUMBER = {2414212},
}

@article{MukherjeeSaha2021,
    AUTHOR = {Mukherjee, Mayukh and Saha, Soumyajit},
     TITLE = {On the effects of small perturbation on low energy {L}aplace eigenfunctions},
   JOURNAL = {J.\ Spectral Theory},
    VOLUME = {15},
      YEAR = {2025},
    NUMBER = {3},
     PAGES = {1045--1087},
}

@article{MukherjeeSaha2022,
    AUTHOR = {Mukherjee, Mayukh and Saha, Soumyajit},
     TITLE = {Nodal sets of {L}aplace eigenfunctions under small perturbations},
   JOURNAL = {Math.\ Ann.},
    VOLUME = {383},
      YEAR = {2022},
    NUMBER = {1--2},
     PAGES = {475--491},
  MRNUMBER = {4444128},
}

@article{Nadirashvili1991,
    AUTHOR = {Nadirashvili, Nikolai S.},
     TITLE = {Metric properties of eigenfunctions of the {L}aplace operator on manifolds},
   JOURNAL = {Ann.\ Inst.\ Fourier (Grenoble)},
    VOLUME = {41},
      YEAR = {1991},
    NUMBER = {1},
     PAGES = {259--265},
  MRNUMBER = {1112199},
}

@article{NazarovPolterovichSodin,
    AUTHOR = {Nazarov, F\"{e}dor and Polterovich, Leonid and Sodin, Mikhail},
     TITLE = {Sign and area in nodal geometry of {L}aplace eigenfunctions},
   JOURNAL = {Amer.\ J.\ Math.},
    VOLUME = {127},
      YEAR = {2005},
    NUMBER = {4},
     PAGES = {879--910},
  MRNUMBER = {2154374},
}

@article{peetre1957generalization,
    AUTHOR = {Peetre, Jaak},
     TITLE = {A generalization of {C}ourant's nodal domain theorem},
   JOURNAL = {Math.\ Scand.},
    VOLUME = {5},
      YEAR = {1957},
     PAGES = {15--20},
  MRNUMBER = {92917},
}

@article{Pleijel1956,
    AUTHOR = {Pleijel, {\AA}ke},
     TITLE = {Remarks on {C}ourant's nodal line theorem},
   JOURNAL = {Comm.\ Pure Appl.\ Math.},
    VOLUME = {9},
      YEAR = {1956},
     PAGES = {543--550},
  MRNUMBER = {80861},
}

@article{polterovich2009pleijel,
    AUTHOR = {Polterovich, Iosif},
     TITLE = {Pleijel's nodal domain theorem for free membranes},
   JOURNAL = {Proc.\ Amer.\ Math.\ Soc.},
    VOLUME = {137},
      YEAR = {2009},
    NUMBER = {3},
     PAGES = {1021--1024},
  MRNUMBER = {2457442},
}

@misc{Ste,
    AUTHOR = {Stern, Antonie},
     TITLE = {Bemerkungen \"{u}ber asymptotisches {V}erhalten von {E}igenwerten und {E}igenfunktionen},
      NOTE = {Math.-naturwiss.\ Diss., G\"{o}ttingen},
      YEAR = {1925},
}

@article{Takahashi2002,
    AUTHOR = {Takahashi, Junya},
     TITLE = {Collapsing of connected sums and the eigenvalues of the {L}aplacian},
   JOURNAL = {J.\ Geom.\ Phys.},
    VOLUME = {40},
      YEAR = {2002},
    NUMBER = {3--4},
     PAGES = {201--208},
  MRNUMBER = {1866987},
}

@article{Uhlenbeck1976,
    AUTHOR = {Uhlenbeck, K.},
     TITLE = {Generic properties of eigenfunctions},
   JOURNAL = {Amer.\ J.\ Math.},
    VOLUME = {98},
      YEAR = {1976},
    NUMBER = {4},
     PAGES = {1059--1078},
  MRNUMBER = {464332},
}

@article{alessandrini1992,
    AUTHOR = {Alessandrini, Giovanni},
     TITLE = {Nodal lines of eigenfunctions of the fixed membrane problem in general convex domains},
   JOURNAL = {Comment.\ Math.\ Helv.},
    VOLUME = {69},
      YEAR = {1994},
    NUMBER = {1},
     PAGES = {142--154},
  MRNUMBER = {1259610},
}

@book{BroeckerJaenich,
    AUTHOR = {Br\"{o}cker, Theodor and J\"{a}nich, Klaus},
     TITLE = {Introduction to differential topology},
      NOTE = {Translated from the {G}erman by {C}.~{B}.\ {T}homas and {M}.~{J}.\ {T}homas},
 PUBLISHER = {Cambridge University Press, Cambridge--New York},
      YEAR = {1982},
     PAGES = {vii+160},
  MRNUMBER = {674117},
}

@article{HartmanWintner1953,
    AUTHOR = {Hartman, Philip and Wintner, Aurel},
     TITLE = {On the local behavior of solutions of non-parabolic partial differential equations},
   JOURNAL = {Amer.\ J.\ Math.},
    VOLUME = {75},
      YEAR = {1953},
     PAGES = {449--476},
  MRNUMBER = {58082},
}

@book{Hirsch1976,
    AUTHOR = {Hirsch, Morris W.},
     TITLE = {Differential topology},
    SERIES = {Graduate Texts in Mathematics, No.\ 33},
 PUBLISHER = {Springer-Verlag, New York--Heidelberg},
      YEAR = {1976},
     PAGES = {x+221},
  MRNUMBER = {448362},
}

@article{Logunov2018,
    AUTHOR = {Logunov, Alexander},
     TITLE = {Nodal sets of {L}aplace eigenfunctions: polynomial upper estimates of the {H}ausdorff measure},
   JOURNAL = {Ann.\ of Math.\ (2)},
    VOLUME = {187},
      YEAR = {2018},
    NUMBER = {1},
     PAGES = {221--239},
  MRNUMBER = {3739231},
}

@article{NaberValtorta2017,
    AUTHOR = {Naber, Aaron and Valtorta, Daniele},
     TITLE = {Volume estimates on the critical sets of solutions to elliptic {PDE}s},
   JOURNAL = {Comm.\ Pure Appl.\ Math.},
    VOLUME = {70},
      YEAR = {2017},
    NUMBER = {10},
     PAGES = {1835--1897},
  MRNUMBER = {3688031},
}

@article{Kom2005,
    AUTHOR = {Komendarczyk, R.},
     TITLE = {On the contact geometry of nodal sets},
   JOURNAL = {Trans.\ Amer.\ Math.\ Soc.},
    VOLUME = {358},
      YEAR = {2006},
    NUMBER = {6},
     PAGES = {2399--2413},
  MRNUMBER = {2204037},
}

@article {Chen1997,
    AUTHOR = {Chen, Xu-Yan},
     TITLE = {Polynomial asymptotics near zero points of solutions of
              general elliptic equations},
   JOURNAL = {Comm. Partial Differential Equations},
  FJOURNAL = {Communications in Partial Differential Equations},
    VOLUME = {22},
      YEAR = {1997},
    NUMBER = {7-8},
     PAGES = {1227--1250},
      ISSN = {0360-5302,1532-4133},
   MRCLASS = {35J45 (35B40)},
  MRNUMBER = {1466314},
MRREVIEWER = {Jan\ Bochenek},
       DOI = {10.1080/03605309708821298},
       URL = {https://doi.org/10.1080/03605309708821298},
}

\end{document}